\documentclass[a4paper,reqno,10pt]{amsart}

\usepackage{a4wide}

\usepackage{amssymb}
\usepackage{amstext}
\usepackage{amsmath}
\usepackage{amscd}
\usepackage{amsthm}
\usepackage{amsfonts}

\usepackage{graphicx}
\usepackage{latexsym}
\usepackage{mathrsfs}

\usepackage{mathtools}
\usepackage[all,poly,necula]{xy}
\usepackage{ifthen}

\xyoption{all}

\usepackage{lscape}

\usepackage{ stmaryrd }
\usepackage{multirow}
\usepackage{array}

\usepackage{enumitem}
\setlist[enumerate,1]{label={\upshape(\arabic*)}}
\setlist[enumerate,2]{label={\upshape(\alph*)}}

\usepackage{tikz}
\usetikzlibrary{arrows,decorations.pathmorphing,decorations.pathreplacing}

\tikzset{vertex/.style={circle,fill=black,inner sep=1.3pt,outer sep=2pt},
         mvertex/.style={rectangle,draw=black,thick,inner sep=2pt,outer sep=2pt},
         cvertex/.style={circle,fill=white,draw=white,thick,inner sep=.5pt,outer sep=.5pt},
         cver/.style={circle,fill=white,draw=black,thick,inner sep=3pt,outer sep=2pt},
         rver/.style={rectangle,fill=white,draw=black,thick,inner sep=2pt,outer sep=2pt},
         tvertex/.style={inner sep=1pt,font=\upshape},
         unvertex/.style={circle,fill=white,draw=white,inner sep=1pt},
         fill1/.style={fill=black!20,draw=black!20},
         fill2/.style={fill=black!40,draw=black!40},
         fill12/.style={fill=black!60,draw=black!60},
         >=stealth',
         leadsto/.style={-angle 90,decorate,decoration=snake,very thick},
         cut/.style={decorate,decoration=saw,very thick}}

\newtheorem{theorem}{Theorem}[section]
\newtheorem{theoremi}{Theorem}

\newtheorem{corollary}[theorem]{Corollary}
\newtheorem{lemma}[theorem]{Lemma}
\newtheorem{proposition}[theorem]{Proposition}
\newtheorem{definition-proposition}[theorem]{Definition-Proposition}

\theoremstyle{definition}
\newtheorem{definition}[theorem]{Definition}

\newtheorem{remark}[theorem]{Remark}
\newtheorem{example}[theorem]{Example}

\newcommand{\CC}{\mathcal{C}}

\newcommand{\DD}{\mathcal{D}}

\newcommand{\KKK}{\mathsf{K}}

\newcommand{\GGG}{\mathsf{G}}
\newcommand{\PP}{\mathcal{P}}
\newcommand{\II}{\mathcal{I}}
\newcommand{\JJ}{\mathcal{J}}

\renewcommand{\SS}{\mathcal{S}}

\newcommand{\XX}{\mathcal{X}}
\newcommand{\XXX}{\mathsf{X}}

\newcommand{\R}{\mathbb{R}}

\newcommand{\rad}{\operatorname{rad}\nolimits}

\newcommand{\Ext}{\operatorname{Ext}\nolimits}

\newcommand{\Hom}{\operatorname{Hom}\nolimits}

\newcommand{\End}{\operatorname{End}\nolimits}

\newcommand{\gl}{\operatorname{gl.\!dim}\nolimits}
\newcommand{\height}{\operatorname{ht}\nolimits}
\newcommand{\dm}{\operatorname{dom.\!dim}\nolimits}
\newcommand{\op}{\operatorname{op}\nolimits}
\newcommand{\RHom}{\mathbf{R}\strut\kern-.2em\operatorname{Hom}\nolimits}

\newcommand{\Image}{\operatorname{Im}\nolimits}
\newcommand{\Kernel}{\operatorname{Ker}\nolimits}
\newcommand{\Cokernel}{\operatorname{Coker}\nolimits}

\newcommand{\Ab}{\mathcal{A}b}
\newcommand{\coker}{\Cokernel}
\newcommand{\im}{\Image}
\renewcommand{\ker}{\Kernel}
\newcommand{\un}{\underline}
\newcommand{\ov}{\overline}

\DeclareMathOperator{\moduleCategory}{\mathsf{mod}} \renewcommand{\mod}{\moduleCategory}
\DeclareMathOperator{\Mod}{\mathsf{Mod}}

\DeclareMathOperator{\proj}{\mathsf{proj}}

\DeclareMathOperator{\ind}{\mathsf{ind}}

\DeclareMathOperator{\Filt}{\mathsf{Filt}}

\DeclareMathOperator{\Ex}{\mathsf{Ex}}
\DeclareMathOperator{\AR}{\mathsf{AR}}

\DeclareMathOperator{\CM}{\mathsf{CM}}
\DeclareMathOperator{\GP}{\mathsf{GP}}

\DeclareMathOperator{\add}{\mathsf{add}}
\DeclareMathOperator{\fl}{\mathsf{f.\!l.}}

\DeclareMathOperator{\pd}{\mathsf{pd}}
\DeclareMathOperator{\id}{\mathsf{id}}
\DeclareMathOperator{\noeth}{\mathsf{noeth}}
\DeclareMathOperator{\art}{\mathsf{art}}

\newcommand{\iso}{\cong}
\newcommand{\infl}{\rightarrowtail}
\newcommand{\defl}{\twoheadrightarrow}

\newcommand{\equi}{\simeq}

\newenvironment{sbmatrix}{\left[\begin{smallmatrix}}{\end{smallmatrix}\right]}

\newsavebox\locboxinminipage
\newlength\locboxinminipagel

\newcommand{\EE}{\mathcal{E}}

\numberwithin{equation}{section}

\xymatrixrowsep{1.5em}

\newdir{ >}{{}*!/-1em/@{>}}
\newcommand{\inflr}{\ar@{ >->}[r]}
\newcommand{\infld}{\ar@{ >->}[d]}
\newcommand{\deflr}{\ar@{->>}[r]}
\newcommand{\defld}{\ar@{->>}[d]}

\begin{document}
\title[Classifications of exact structures]{Classifications of exact structures and\\ Cohen-Macaulay-finite algebras}

\author[H. Enomoto]{Haruhisa Enomoto}
\address{Graduate School of Mathematics, Nagoya University, Chikusa-ku, Nagoya. 464-8602, Japan}
\email{m16009t@math.nagoya-u.ac.jp}
\begin{abstract}
 We give a classification of all exact structures on a given idempotent complete additive category. Using this, we investigate the structure of an exact category with finitely many indecomposables. We show that the relation of the Grothendieck group of such a category is generated by AR conflations. Moreover, we obtain an explicit classification of (1) Gorenstein-projective-finite Iwanaga-Gorenstein algebras, (2) Cohen-Macaulay-finite orders, and more generally, (3) cotilting modules $U$ with $^\perp U$ of finite type. In the appendix, we develop the AR theory of exact categories over a noetherian complete local ring, and relate the existence of AR conflations to the AR duality and dualizing varieties.
\end{abstract}

\maketitle

\tableofcontents

\section{Introduction}
In the representation theory of finite-dimensional algebras, one of the most important subjects is to \emph{classify certain categories of finite type}. Here we say that an additive $k$-category over a field $k$ is \emph{of finite type} if it has only finitely many indecomposable objects up to isomorphism. The aim of this paper is to give a classification of exact categories of finite type, and thereby provide an explicit classification of all GP-finite Iwanaga-Gorenstein algebras. Let us explain the motivation for this.

First we recall how categories of finite type have been studied in the representation theory. It is well-known that an abelian Hom-finite $k$-category of finite type is nothing but the category $\mod\Lambda$ of finitely generated $\Lambda$-modules over some representation-finite $k$-algebra $\Lambda$ (see \cite[8.2]{gabind} or Proposition \ref{enough} below). A classification of such algebras is one of the main problems in the representation theory of algebras, and has been studied widely by a number of papers, e.g. \cite{gabriel,riedtmann,bgrs,gr}.
For the case of representation-finite $R$-orders over a noetherian local ring $R$, we refer the reader to \cite{artin,dk,hn,lw,rvdb2,yo}.
Besides abelian categories, triangulated categories of finite type also has been investigated, e.g. in \cite{amiot}. Such triangulated categories naturally arise in the representation of algebras and in the categorification of cluster algebras.

Among other things, the observation by Auslander \cite{repdim} is of particular importance to us when we deal with categories of finite type. Let $\EE$ be a Hom-finite $k$-category of finite type and consider the algebra $\Gamma := \End_\EE(M)$, where $M$ is a direct sum of all non-isomorphic indecomposables in $\EE$. This $\Gamma$ is called an \emph{Auslander algebra} of $\EE$, and categorical properties of $\EE$ should be related to homological properties of $\Gamma$. For example, the condition $\EE$ being abelian is equivalent to a certain homological condition of $\Gamma$, that is, $\gl\Gamma \leq 2 \leq \dm \Gamma$. This is called the \emph{Auslander correspondence}, and is now the basic and important viewpoint in the representation theory.

However, most of the algebras are representation-wild, and it is hopeless to understand the whole structure of the module category. Thus nice subcategories of module categories has attracted much attention. For example, in the representation theory of a commutative noetherian local ring $R$, the category $\CM R$ of Cohen-Macaulay modules has played a central role. For an Iwanaga-Gorenstein ring (two-sided noetherian ring satisfying $\id(\Lambda_\Lambda) = \id ({}_\Lambda \Lambda ) < \infty$), we have the natural notion of Cohen-Macaulay $\Lambda$-modules (which we call \emph{Gorenstein-projective} $\Lambda$-modules to avoid any confusion). In the spirit of Auslander's observation, it is natural to ask whether there exists an Auslander-type correspondence for \emph{GP-finite} Iwanaga-Gorenstein algebras (algebras $\Lambda$ such that $\GP \Lambda$ is of finite type). In this paper, we give an answer to this problem.

It is known that an algebra $\Gamma$ is the GP-Auslander-algebra of some Iwanaga-Gorenstein algebra if and only if the global dimension of $\Gamma$ is finite (see e.g. Proposition \ref{eno} below or \cite[Proposition 4]{kalck}).
In fact, if $\Gamma$ has finite global dimension, then $\Gamma$ is Iwanaga-Gorenstein and satisfies $\GP \Gamma = \proj\Gamma$. Thus $\Gamma$ is GP-finite and $\Gamma$ is its own GP-Auslander algebra. However, there may exist another GP-finite algebra $\Lambda$ which satisfies $\GP\Lambda \equi \GP\Gamma$ (consider for example a representation-finite selfinjective algebra $\Lambda$ and its usual Auslander algebra $\Gamma$). Therefore, the GP-Auslander-algebra alone is not enough to classify GP-finite Iwanaga-Gorenstein algebras.

In Theorem \ref{theorema}, we classify GP-finite Iwanaga-Gorenstein algebras by algebras with finite global dimension together with their modules satisfying a certain condition, given as follows.
\begin{definition}\label{2regdef}
Let $\Gamma$ be a two-sided noetherian ring and $S$ a simple right $\Gamma$-module. We say that $S$ satisfies the \emph{$2$-regular condition} if the following conditions are satisfied.
\begin{enumerate}
\item $\pd S_\Gamma = 2$.
\item $\Ext^i_\Gamma (S,\Gamma) =0$ for $i =0,1$.
\item $\Ext^2_\Gamma (S,\Gamma)$ is a simple left $\Gamma$-module.
\end{enumerate}
\end{definition}
This condition is satisfied by the simple modules over commutative regular local rings of dimension $2$, and their non-commutative analogues, e.g. Artin-Schelter regular algebras, Calabi-Yau algebras, non-singular orders of dimension $2$.

For a finite-dimensional algebra $\Gamma$, we interpret the $2$-regular condition in terms of the translation quiver $Q(\Gamma)$. This is the usual quiver of $\Gamma$ together with dotted arrows corresponding to the simple $\Gamma$-modules with the $2$-regular condition (Definition \ref{const}). Using these concepts, we obtain the following classification, where a $\Gamma$-module is called \emph{basic} if its indecomposable direct summands are pairwise non-isomorphic.
\begin{theoremi}[\textup{= Corollary \ref{main2}}]\label{theorema}
There exists a bijection between the following for a field $k$.
\begin{enumerate}
\item Morita equivalence classes of GP-finite Iwanaga-Gorenstein finite-dimensional $k$-algebras $\Lambda$.

\item Equivalence classes of pairs $(\Gamma,M)$, where $\Gamma$ is a finite-dimensional $k$-algebra with finite global dimension and $M$ is a basic semisimple $\Gamma$-module such that every simple summand of $M$ satisfies the $2$-regular condition and $\Hom_k(M,k) \iso \Ext_\Gamma^2(M,\Gamma)$ holds as left $\Gamma$-modules.

\item Equivalence classes of pairs $(\Gamma,\XXX)$, where $\Gamma$ is a finite-dimensional $k$-algebra with finite global dimension and $\XXX$ is a union of stable $\tau$-orbits in $Q(\Gamma)$.
\end{enumerate}
For a pair $(\Gamma, \XXX)$ in $(3)$, the corresponding algebra $\Lambda$ in $(1)$ is given by the endomorphism ring of the direct sum of the indecomposable projective $\Gamma$-modules which does not belong to $\tau$-orbits in $\XXX$.
\end{theoremi}
As Example \ref{example1} illustrates, it provides a systematic method to construct a family of GP-finite Iwanaga-Gorenstein algebras.

To prove Theorem \ref{theorema}, we need the notion of \emph{exact categories}, which was introduced by Quillen \cite{qu} as a generalization of abelian categories. It provides an appropriate framework for a relative homological algebra and has a number of applications in many branches of mathematics, such as representation theory, algebraic topology and functional analysis.
Let us explain why the use of exact structures is essential in our classification. In the classical situation, we can recover $\Lambda$ (up to Morita equivalence) only from the additive structure of $\mod\Lambda$, while this fails to be true for $\GP \Lambda$ as mentioned above. Even in this case, $\Lambda$ can be recovered from the \emph{exact structure} of $\GP \Lambda$, as the endomorphism ring of the progenerator of it. In Theorem \ref{theorema}, the algebra $\Gamma$ gives the additive structure of $\GP \Lambda$, while the module $M$ or the set $\XXX$ gives the exact structure on it. In this setting, the quiver $Q(\Gamma)$ with dotted arrow $\XXX$ is nothing but the Auslander-Reiten quiver of $\GP \Lambda$.

\begin{example}\label{example1}
Let $\Gamma$ be the algebra given by the quiver in Figure \ref{cap1},
where we identify two vertical arrows, with commutativity and zero relations indicated by dotted lines. Then $\Gamma$ has finite global dimension and the translation quiver $Q(\Gamma)$ is the same as Figure \ref{cap1}. It has two stable $\tau$-orbits $A$ and $B$, hence we obtain four GP-finite Iwanaga-Gorenstein algebras, corresponding to the endomorphism rings of vertices which does not belong to $\XXX$. Table \ref{cap2} is an explicit calculation of all GP-finite Iwanaga-Gorenstein algebras $\Lambda$ such that $\GP \Lambda$ is equivalent to $\proj \Gamma$.
\begin{figure}[h]
\centering

\begin{tikzpicture}[scale=.65,yscale=-.8]

 \draw [->, very thick] (0,3.5) -- (0,-0.5);
 \draw [->, very thick] (6,3.5) -- (6,-0.5);

 \draw [dashed] (2,3) -- (4,3);
 \draw [dashed] (0,2) -- (6,2);
 \draw [dashed] (0,1) -- (6,1);

 \node (213) at (0,1) [cvertex] {1};
 \node (3213) at (1,0) [cvertex] {2};
 \node (21) at (1,2) [cvertex] {3};
 \node (321) at (2,1) [cvertex] {4};
 \node (2) at (2,3) [cvertex] {5};
 \node (1321) at (3,0) [cvertex] {6};
 \node (32) at (3,2) [cvertex] {7};
 \node (132) at (4,1) [cvertex] {8};
 \node (3) at (4,3) [cvertex] {9};
 \node (2132) at (5,0) [cvertex] {10};
 \node (13) at (5,2) [cvertex] {11};
 \node (n213) at (6,1) [cvertex] {1};

 \node at (-3,1) {(stable $\tau$-orbit $A$)};
 \node at (-3,2) {(stable $\tau$-orbit $B$)};
 \node at (-3.5,3) {(non-stable $\tau$-orbit $C$)};

 \draw [->] (21) -- (2);
 \draw [->] (2) -- (32);
 \draw [->] (32) -- (3);
 \draw [->] (3) -- (13);
 \draw [->] (213) -- (21);
 \draw [->] (21) -- (321);
 \draw [->] (321) -- (32);
 \draw [->] (32) -- (132);
 \draw [->] (132) -- (13);
 \draw [->] (13) -- (n213);
 \draw [->] (213) -- (3213);
 \draw [->] (3213) -- (321);
 \draw [->] (321) -- (1321);
 \draw [->] (1321) -- (132);
 \draw [->] (132) -- (2132);
 \draw [->] (2132) -- (n213);
\end{tikzpicture}
\caption{The quiver of $\Gamma$ (or equivalently, $\proj\Gamma$).}\label{cap1}
\end{figure}

\begin{table}[t]

\begin{tabular}{c|c|c}
$\XXX$ & Quivers for $\Lambda$ & Relations for $\Lambda$ \\ \hline\hline
$\varnothing$ & $\Gamma$ in Figure \ref{cap1} & The same relation as $\Gamma$. \\ \hline
$A \cup B$ &
\begin{tikzpicture}[scale=.64,yscale=-1,baseline=(current bounding box.center)]

 \node (9) at (-.5,.5) [tvertex] {9};
 \node (2) at (1,.5) [tvertex] {2};
 \node (6) at (2,-.5) [tvertex] {6};
 \node (10) at (3,.5) [tvertex] {10};
 \node (5) at (4.5,.5) [tvertex] {5};
 \node  at (2,-1) [tvertex] {};

 \draw [->] (2) -- (6) node [near start,above] {$a$};
 \draw [->] (6) -- (10) node [near end,above] {$b$};
 \draw [->] (10) -- (2) node [midway,below] {$c$};
 \draw [->] (2) to [bend left] node [midway,above] {$d$} (9) ;
 \draw [->] (9) to [bend left] node [midway,below] {$e$}(2);
 \draw [->] (10) to [bend right] node [midway,above] {$f$}(5);
 \draw [bend right,->] (5) to [bend right] node [midway,below] {g} (10);

\end{tikzpicture}
 &
 $\begin{aligned}[c]
 c b a &= e d, \\
 b a c &= g f, \\
 f b = a e = a c b = d c &= f g = c g = d e = 0.
 \end{aligned}$
 \\ \hline
$A$ &
\begin{tikzpicture}[scale=.6,yscale=-.8,baseline=(current bounding box.center)]

 \draw [->, very thick] (0,3) -- (0,-.5);
 \draw [->, very thick] (6,3) -- (6,-.5);

 \node (2) at (0,0.5) [cvertex] {2};
 \node (3) at (0,2) [cvertex] {3};
 \node (5) at (1,3) [cvertex] {5};
 \node (6) at (2,0.5) [cvertex] {6};
 \node (7) at (2,2) [cvertex] {7};
 \node (9) at (3,3) [cvertex] {9};
 \node (10) at (4,0.5) [cvertex] {10};
 \node (11) at (4,2) [cvertex] {11};
 \node (n2) at (6,0.5) [cvertex] {2};
 \node (n3) at (6,2) [cvertex] {3};

 \draw [->] (2) -- (7);
 \draw [->] (7) -- (10);
 \draw [->] (10) -- (n3);
 \draw [->] (3) -- (6);
 \draw [->] (6) -- (11);
 \draw [->] (11) -- (n2);
 \draw [->] (3) -- (5);
 \draw [->] (5) -- (7);
 \draw [->] (7) -- (9);
 \draw [->] (9) -- (11);
 \draw [->] (2) -- (6);
 \draw [->] (6) -- (10);
 \draw [->] (10) -- (n2);

\end{tikzpicture}
 &
 $\begin{aligned}
 \text{Obvious commutative relations,}& \text{ and the zero relations of} \\
 5 \to 7 \to & 9, \\
 \text{two paths of length two}&\text{ from $11$,} \\
 \text{two paths of length two}&\text{ to $3$.}
 \end{aligned}$
 \\ \hline
$B$ &
\begin{tikzpicture}[scale=.6,yscale=-.8,baseline=(current bounding box.center)]

 \draw [->, very thick] (0,3) -- (0,-0.5);
 \draw [->, very thick] (6,3) -- (6,-0.5);

 \node (1) at (0,1) [cvertex] {1};
 \node (2) at (1,0) [cvertex] {2};
 \node (4) at (2,1) [cvertex] {4};
 \node (5) at (2,2.5) [cvertex] {5};
 \node (6) at (3,0) [cvertex] {6};
 \node (8) at (4,1) [cvertex] {8};
 \node (9) at (4,2.5) [cvertex] {9};
 \node (10) at (5,0) [cvertex] {10};
 \node (n1) at (6,1) [cvertex] {1};

 \draw [->] (1) -- (2);
 \draw [->] (2) -- (4);
 \draw [->] (4) -- (6);
 \draw [->] (6) -- (8);
 \draw [->] (8) -- (10);
 \draw [->] (10) -- (n1);
 \draw [->] (1) -- (5);
 \draw [->] (5) -- (8);
 \draw [->] (4) -- (9);
 \draw [->] (9) -- (n1);

\end{tikzpicture}
 & $\begin{aligned}
 \text{Obvious commutative relations,}& \text{ and the zero relations of} \\
 5 \to 8 \to &10 \to 1, \\
 1 \to 2 \to &4 \to 9, \\
 \text{two paths of length three and}&\text{ four from $8$ to $5$ and $4$,} \\
 \text{two paths of length two and}&\text{ three from $9$ to $5$ and $4$.}
 \end{aligned}$  \\ \hline
\end{tabular}
\caption{All CM-finite Iwanaga-Gorenstein algebras with its CM category $\proj\Gamma$.}\label{cap2}
\end{table}

\end{example}

To deal with exact structures, we mainly work with idempotent complete additive categories instead of categories of finite type. For such a category $\EE$, we classify all exact structures on $\EE$ in terms of  its functor category $\mod\EE$. The precise statement is the following, where $\CC_2(\EE)$ is the category of $\EE$-modules whose projective dimension and grade are equal to $2$.

\begin{theoremi}[\textup{= Theorem \ref{main}}]\label{theoremb}
Let $\EE$ be an idempotent complete additive category. Then there exists a bijection between the following two classes.
\begin{enumerate}
\item Exact structures on $\EE$.
\item Subcategories $\DD$ of $\CC_2(\EE)$ satisfying the following conditions.
\begin{enumerate}
\item $\DD$ is a Serre subcategory of $\mod\EE$.
\item $\Ext^2_\EE(\DD,\EE)$ is a Serre subcategory of $\mod \EE^{\op}$.
\end{enumerate}
\end{enumerate}
\end{theoremi}
It is surprising to us that such a general description of all exact structures is available. For example, the existence theorem of the largest exact structure due to \cite{rump2} easily follows from Theorem \ref{theoremb} in the case of idempotent complete categories.

We apply Theorem \ref{theoremb} to exact categories of finite type, which is our motivating object. When $\EE$ is a Hom-finite idempotent complete $k$-category of finite type, we show that exact structures on $\EE$ bijectively correspond to sets of simple $\Gamma$-modules satisfying the $2$-regular condition, and sets of dotted arrows of $Q(\Gamma)$, where $\Gamma$ is the Auslander algebra of $\EE$ (Theorem \ref{2regularcorresp} and Corollary \ref{dotted}). We apply this result to obtain Theorem \ref{theorema}, and more generally, an Auslander correspondence for cotilting modules $U$ such that $^\perp U$ is of finite type (Theorem \ref{main1}).
As another application, we obtain the following Auslander correspondence for representation-finite $R$-orders in case $\dim R \geq 2$. This  improves \cite[Theorem 4.2.3]{higher} for the case $n=1$, because our result gives a bijection for representation-finite orders and some assumptions on $\Gamma$ in \cite[Theorem 4.2.3]{higher} are dropped.
\begin{theoremi}[\textup{= Corollary \ref{arorder}}]
Let $R$ be a complete Cohen-Macaulay local ring with $\dim R =d  \geq 2$. Then there exists a bijection between the following.
\begin{enumerate}
\item Morita equivalence classes of $R$-orders $\Lambda$ such that $\CM \Lambda$ is of finite type.
\item Equivalence classes of pairs $(\Gamma, e)$, where $\Gamma$ is a noetherian $R$-algebra and $e$ is an idempotent of $\Gamma$ such that the following conditions are satisfied.
\begin{enumerate}
\item $\gl \Gamma = d$.
\item $\Gamma e$ is maximal Cohen-Macaulay as an $R$-module.
\item $\un{\Gamma}:= \Gamma / \Gamma e \Gamma$ is of finite length over $R$.
\item $\un{\Gamma}/ (\rad \un{\Gamma})$ is a direct sum of simple $\Gamma$-modules satisfying the $2$-regular condition.
\end{enumerate}
\end{enumerate}
\end{theoremi}

Also, we investigate the Grothendieck group $\KKK_0(\EE)$ of an exact category $\EE$ of finite type, and show that the relation of $\KKK_0(\EE)$ is generated by AR conflations under some mild conditions (Corollary \ref{grocor}). This unifies several known results by \cite{ar1,but,yo}.

The appendix contains a brief discussion of the Auslander-Reiten theory for exact categories which are Hom-noetherian over a noetherian complete local ring $R$. We show that the existence of AR conflations is closely related to the AR duality and dualizing $R$-varieties (Theorem \ref{confldual}). These observations shed new light on the result on isolated singularities by Auslander \cite{isolated}, see Remark \ref{isolatedrmk}.

This paper is organized as follows. In Section 2, we state and prove our main classification. In Section 3, we investigate exact categories of finite type. In Section 4, we  apply our previous results to the study of the representation theory of algebras.

\subsection{Conventions}
Throughout this paper, \emph{we assume that all categories are skeletally small}, that is, the isomorphism classes of objects form a set. In addition, \emph{all subcategories are assumed to be full and closed under isomorphisms}. By a \emph{module} we always mean right modules unless otherwise stated.
As for exact categories, we use the terminologies \emph{inflations}, \emph{deflations} and \emph{conflations}. We refer the reader to \cite{buhler} for the basics of exact categories.

\section{Classifying exact structures via Serre subcategories}
In this section, we give a bijection between exact structures on an idempotent complete additive category $\EE$ and subcategories of $\mod\EE$ satisfying certain ``2-regular condition.''
\emph{Throughout this section, we always assume that $\EE$ is an additive category}.

Recall that an additive category $\EE$ is \emph{idempotent complete} if every morphism $e:X \to X$ in $\EE$ satisfying $e^2 = e$ has a kernel, or equivalently, a cokernel. For example, every subcategory of an abelian category which is closed under direct sums and direct summands are idempotent complete.

\subsection{Preliminaries on functor categories}
Our strategy to investigate exact structures on $\EE$ is to study the \emph{module category} $\Mod\EE$ over $\EE$. Let us recall related definitions.

For an additive category $\EE$, a \emph{right $\EE$-module} $M$ is a contravariant additive functor $M: \EE^{\op} \to \Ab$ from $\EE$ to the category of abelian groups $\Ab$. We denote by $\Mod \EE$ the category of right $\EE$-modules, and morphisms in $\Mod \EE$ are natural transformations between them. This category $\Mod\EE$ is an abelian category with enough projectives, and projective objects are precisely direct summands of (possibly infinite) direct sums of representable functors.
For simplicity, we put $P_X := \EE(-,X) \in \Mod\EE$ and $P^X := \EE(X,-) \in \Mod\EE^{\op}$ for any $X$ in $\EE$. Note that $P_{(-)}:\EE \to \Mod\EE$ gives the Yoneda embedding.

We say that a $\EE$-module $M$ is \emph{finitely generated} if there exists an epimorphism $P_X \defl M$ for some $X$ in $\EE$. Throughout this paper, we often use the fact that \emph{$\EE$ is idempotent complete if and only if the essential image of the Yoneda embedding $P_{(-)}: \EE \to \Mod\EE$ consists of all finitely generated projective $\EE$-modules}. We denote by $\mod \EE$ the category of finitely generated $\EE$-modules. We also denote by $\mod_1 \EE$ the category of finitely presented $\EE$-modules, that is, the modules $M$ such that there exists an exact sequence $P_X \to P_Y \to M \to 0$ for some $X,Y$ in $\EE$.

We define a contravariant functor $\Hom_\EE(-,\EE): \Mod\EE \to \Mod \EE^{\op}$ by the following way: For $M$ in $\Mod\EE$, the left $\EE$-module $\Hom_\EE(M,\EE): \EE \to \Ab$ is the composition of the Yoneda embedding $\EE \to \Mod\EE$ and $(\Mod\EE)(M,-):\Mod\EE \to \Ab$. Note that $\Hom_\EE(-,\EE)$ is a left exact functor satisfying $\Hom_\EE(P_{(-)},\EE) \equi P^{(-)}$. We denote by $\Ext_\EE^i(-,\EE) :\Mod\EE \to \Mod\EE^{\op}$ the $i$-th right derived functor of $\Hom_\EE(-,\EE)$.

Using these concepts, we can interpret kernel-cokernel pairs in terms of modules over $\EE$.
Here we say that a complex
$
X \xrightarrow{f} Y \xrightarrow{g} Z
$
in $\EE$ is a \emph{kernel-cokernel pair} if $f$ is a kernel of $g$ and $g$ is a cokernel of $f$.
\begin{proposition}\label{aiu}
Let $X \xrightarrow{f} Y \xrightarrow{g} Z$ be a complex in $\EE$. Put $M:= \coker(P_g) \in \Mod\EE$. Then the following hold.
\begin{enumerate}
\item $g$ is an epimorphism in $\EE$ if and only if $\Hom_\EE(M,\EE)=0$.
\item $X \xrightarrow{f} Y \xrightarrow{g} Z$ is a kernel-cokernel pair if and only if the following conditions are satisfied.
\begin{enumerate}
\item $0 \to P_X \xrightarrow{P_f} P_Y \xrightarrow{P_g} P_Z \to M \to 0$ is exact.
\item $\Ext_\EE^i (M,\EE)=0$ for $i = 0,1$.
\end{enumerate}
\end{enumerate}
\end{proposition}
\begin{proof}
(1)
We have an exact sequence $0 \to \Hom_\EE(M,\EE) \to P^Z \xrightarrow{P^g} P^Y$. Thus $g$ is an epimorphism if and only if $P^g $ is a monomorphism if and only if $\Hom_\EE(M,\EE)=0$.

(2)
The condition (a) is equivalent to that $f$ is a kernel of $g$. Under (a), the condition (b) is equivalent to that $0 \to P^Z \xrightarrow{P^g} P^Y \xrightarrow{P^f} P^X$ is exact, which holds precisely when $g$ is a cokernel of $f$.
\end{proof}

We need the following technical lemma later.
\begin{lemma}\label{Schanuel}
Suppose that $\EE$ is idempotent complete and there exists an exact sequence
\[
0 \to P_X \to P_Y \to P_Z \to M \to 0
\]
in $\Mod\EE$. For any morphism $h:B \to C$ with $\coker (P_h) \iso M$, there exists an object $A$ in $\EE$ such that $\ker P_h \iso P_A$.
\end{lemma}
\begin{proof}
Schanuel's lemma shows that $\ker P_h \oplus P_Y \oplus P_C \iso P_X \oplus P_B \oplus P_Z$, which clearly implies that $\ker P_h$ is finitely generated projective. Since $\EE$ is idempotent complete, the assertion holds.
\end{proof}

\subsection{Construction of maps}
First we fix some notations which we need to describe our main theorem. The following observation follows immediately from Proposition \ref{aiu}.
\begin{lemma}\label{kctokuchou}
For an object $M$ in $\Mod\EE$, the following are equivalent.
\begin{enumerate}
\item There exists a kernel-cokernel pair $X \to Y \xrightarrow{g} Z$ in $\EE$ such that $M$ is isomorphic to $\coker (P_g)$.
\item There exists an exact sequence $0 \to P_X \to P_Y \to P_Z \to M \to 0$ in $\Mod\EE$ and $\Ext_\EE^i(M,\EE)=0$ for $i=0,1$.
\end{enumerate}
\end{lemma}
We denote by $\CC_2(\EE)$ the subcategory of $\Mod\EE$ consisting of $\EE$-modules satisfying the above equivalent conditions. This class of modules play an indispensable role throughout this paper.

\begin{lemma}\label{summandclosed}
The category $\CC_2(\EE)$ is closed under direct summands in $\Mod\EE$.
\end{lemma}
\begin{proof}
We denote by $\mod_2 \EE$ the subcategory of $\Mod\EE$ consisting of all objects $M$ such that there exists an exact sequence
\[
P_2 \to P_1 \to P_0 \to M \to 0
\]
in $\Mod\EE$, where each $P_i$ is finitely generated projective.

In \cite[Lemma 2.5(a)]{kimura}, it was shown that $\mod_2\EE$ is closed under summands in $\Mod\EE$.
Note that $\CC_2(\EE)$ is equal to the intersection of two subcategories of $\mod_2\EE$:
\begin{enumerate}
\item the subcategory consisting of all objects whose projective dimensions are at most $2$ and
\item the subcategory consisting of all objects $M$ such that $\Ext^i_\EE(M,\EE)=0$ for $i=0,1$.
\end{enumerate}
 Since these subcategories are closed under direct summands in $\Mod\EE$, so is $\CC_2(\EE)$.
\end{proof}

\begin{lemma}\label{e2}
The functor $\Ext^2_\EE(-,\EE)$ induces a duality of exact categories $\CC_2(\EE) \to \CC_2(\EE^{\op})$.
\end{lemma}
\begin{proof}
The same proof as in \cite[6.2(1)]{tau3} applies here.
\end{proof}

By definition, the category $\CC_2(\EE)$ is closely related to kernel-cokernel pairs of $\EE$. Indeed, we have two maps between them, as we shall construct below.

Suppose that $F$ is a class of kernel-cokernel pairs in $\EE$. We say that a complex $X \xrightarrow{f} Y \xrightarrow{g} Z$ is an \emph{$F$-exact} if it belongs to $F$. In this case, we say that $f$ is an \emph{$F$-monomorphism} and that $g$ is an \emph{$F$-epimorphism}.

\begin{definition}\label{maps}
\begin{enumerate}[leftmargin=*]
\item
 Let $\DD$ be a subcategory of $\CC_2(\EE)$. We denote by $F(\DD)$ the class of all complexes $X \xrightarrow{f} Y \xrightarrow{g} Z$ which satisfy the following condition:\\
 \emph{There exists an exact sequence $0 \to P_X \xrightarrow{P_f} P_Y \xrightarrow{P_g} P_Z \to M \to 0$ in $\Mod\EE$ such that $M$ belongs to $\DD$.}
 \item
 Let $F$ be a class of kernel-cokernel pairs in $\EE$. We denote by $\DD(F)$ the subcategory of $\Mod\EE$ consisting of all objects $M$ which satisfy the following condition:\\
 \emph{There exists an $F$-exact complex $X \xrightarrow{f} Y \xrightarrow{g} Z$ satisfying $M \iso \coker (P_g)$.}
\end{enumerate}
\end{definition}
The category $\DD(F)$ can be seen as the category of \emph{contravariant defects} of $F$-exact sequences, in the sense of Auslander (see \cite[IV.4]{ars}).
Note that by Lemma \ref{kctokuchou}, the following hold.
\begin{enumerate}
\item[$\bullet$] Every complex $X \xrightarrow{f} Y \xrightarrow{g} Z$ in $F(\DD)$ is a kernel-cokernel pair.
\item[$\bullet$] Every object $M$ in $\DD(F)$ is  in $\CC_2(\EE)$.
\end{enumerate}

\subsection{Main theorem}
In this subsection, we will state our main theorem and give a proof.
Recall that a \emph{Serre subcategory} of an exact category $\EE$ is an additive subcategory $\DD$  of $\EE$ such that for all conflations $L \infl M \defl N$ in $\EE$, the object $M$ belongs to $\DD$ if and only if both $L$ and $N$ belong to $\DD$.

\begin{theorem}\label{main}
Let $\EE$ be an idempotent complete additive category. Then there exist mutually inverse bijections between the following two classes.
\begin{enumerate}
 \item Exact structures $F$ on $\EE$.
 \item Subcategories $\DD$ of $\CC_2(\EE)$ satisfying the following conditions.
 \begin{enumerate}
  \item $\DD$ is a Serre subcategory of $\mod\EE$.
  \item $\Ext^2_\EE(\DD,\EE)$ is a Serre subcategory of $\mod\EE^{\op}$.
 \end{enumerate}
 \end{enumerate}
 The map from $(1)$ to $(2)$ is given by $F \mapsto \DD(F)$ and from $(2)$ to $(1)$ by $\DD \mapsto F(\DD)$ (see Definition \ref{maps}).
\end{theorem}

To prove this, we first show that the maps in Definition \ref{maps} induces a one-to-one correspondence between wider classes than in Theorem \ref{main}.

\begin{proposition}\label{bij}
Let $\EE$ be an additive category. Then the maps in Definition \ref{maps} induce mutually inverse bijections between the following two classes.
\begin{enumerate}
 \item Classes of kernel-cokernel pairs $F$ in $\EE$ satisfying the following conditions.
 \begin{enumerate}
  \item $F$ is closed under homotopy equivalences of complexes.
  \item $F$ is closed under direct sums of complexes.
  \item $F$ is closed under direct summands of complexes.
  \item $F$ is not empty.
 \end{enumerate}
 \item Subcategories $\DD$ of $\CC_2(\EE)$ satisfying the following condition.
 \begin{enumerate}
  \item $\DD$ is closed under direct sums.
  \item $\DD$ is closed under direct summands.
  \item $\DD$ is not empty.
 \end{enumerate}
 \end{enumerate}
\end{proposition}
\begin{proof}
First we see that the maps in Definition \ref{maps} induces well-defined maps between (1) and (2).

(1) $\to$ (2):
Suppose that $F$ is a class of kernel-cokernel pairs in $\EE$ satisfying the conditions of (1). We will see that $\DD(F)$ satisfies the conditions of (2).
Clearly $\DD(F)$ satisfies (a) and (c) by the conditions (1)(b) and (1)(d) respectively. Hence it suffices to show that $\DD$ is closed under direct summands.

Suppose that $M_1 \oplus M_2$ is in $\DD(F)$. Then $M_1 \oplus M_2$ is an object of $\CC_2(\EE)$, which is closed under summands by Lemma \ref{summandclosed}. Therefore both $M_1$ and $M_2$ are in $\CC_2(\EE)$. This gives exact sequences
\[
0 \to P_{X_i} \xrightarrow{P_{f_i}} P_{Y_i} \xrightarrow{P_{g_i}} P_{Z_i} \to M_i \to 0
\]
for some kernel-cokernel pairs $X_i \xrightarrow{f_i} Y_i \xrightarrow{g_i} Z_i$ in $\EE$ with $i=1,2$. By taking the direct sum of these two complexes, we obtain a kernel-cokernel pair
\begin{equation}\label{areare1}
X_1\oplus X_2 \xrightarrow{f_1 \oplus f_2} Y_1\oplus Y_2 \xrightarrow{g_1 \oplus g_2} Z_1 \oplus Z_2.
\end{equation}
Thereby we obtain the following exact sequence.
\[
0 \to P_{X_1\oplus X_2} \xrightarrow{P_{f_1}\oplus P_{f_2}} P_{Y_1\oplus Y_2} \xrightarrow{P_{g_1} \oplus P_{g_2}} P_{Z_1\oplus Z_2} \to M_1 \oplus M_2 \to 0
\]
On the other hand, since $M_1 \oplus M_2$ is in $\DD(F)$, there exists an $F$-exact complex
\begin{equation}\label{areare2}
X \xrightarrow{f} Y \xrightarrow{g} Z
\end{equation} in $\EE$ such that
\[
0 \to P_X \xrightarrow{P_f} P_Y \xrightarrow{P_g} P_Z \to M_1 \oplus M_2 \to 0
\]
is exact. It is standard that two projective resolutions are homotopy equivalent to each other. Thus the Yoneda lemma implies that (\ref{areare1}) and (\ref{areare2}) are homotopy equivalent to each other as three-term complexes in $\EE$. Because $F$ is closed under homotopy equivalences, (\ref{areare1}) is $F$-exact. By the condition (1)(c), we conclude that each $g_i$ is an $F$-epimorphism, which shows that $M_1$ and $M_2$ are in $\DD(F)$.

(2) $\to$ (1):
This is immediate and we leave the details to the reader.

Now we will show that the maps in Definition \ref{maps} are in fact inverse to each other.
It is easy to see that $\DD = \DD(F(\DD))$ in general. (Note that \emph{subcategories} are always assumed to be closed under isomorphisms.)

Suppose $F$ is a class of kernel-cokernel pairs in $\EE$ satisfying the conditions of $(1)$. Put $\DD := \DD(F)$. Then clearly we have $F(\DD) \supset F$. We will prove $F(\DD) \subset F$.

Suppose that we have an $F(\DD)$-exact sequence
\begin{equation}\label{areare3}
X \xrightarrow{f} Y \xrightarrow{g} Z.
\end{equation}
Put $M:= \coker(P_g)$. Then $M$ is contained in $\DD$. Thus there exists an $F$-exact sequence
\begin{equation}\label{areare4}
X' \xrightarrow{f'} Y' \xrightarrow{g'} Z'
\end{equation}
such that
\[
0 \to P_{X'} \xrightarrow{P_{f'}} P_{Y'} \xrightarrow{P_{g'}} P_{Z'} \to M \to 0
\]
is exact. Since the Yoneda embedding of (\ref{areare3}) also yields a projective resolution for $M$, it follows that (\ref{areare3}) and (\ref{areare4}) are homotopy equivalent to each other. Therefore (\ref{areare3}) is an $F$-exact sequence, since $F$ is closed under homotopy equivalences.
\end{proof}

Now we are in the position to prove Theorem \ref{main}. First we recall the axiom of exact categories following \cite[Definition 2.1]{buhler}.
Let $F$ be a class of kernel-cokernel pairs in $\EE$ which is closed under isomorphisms. Consider the following conditions.
\begin{enumerate}
\item[(Ex0)] For all objects $X \in\EE$, the complex $0\to X = X$ is $F$-exact.
\item[(Ex1)] The class of $F$-epimorphisms are closed under compositions.
\item[(Ex2)] The pullback of an $F$-epimorphisms along an arbitrary morphism exists and yields an $F$-epimorphism.
\end{enumerate}
We say that \emph{$(\EE,F)$ is an exact category} if $F$ satisfies (Ex0)-(Ex2) and (Ex0)$^{\op}$-(Ex2)$^{\op}$. In this case, $F$ is called an \emph{exact structure} on $\EE$, and we just call $\EE$ an exact category if $F$ is clear from context. In an exact category $(\EE,F)$, we use the terminologies \emph{conflations}, \emph{deflations} and \emph{inflations} instead of $F$-exact sequences, $F$-epimorphisms and $F$-monomorphisms respectively.

\begin{lemma}\label{ok}
Let $(\EE,F)$ be an exact category. Then $F$ satisfies all the conditions of Proposition \ref{bij}(1).
\end{lemma}
\begin{proof}
The condition (1)(d) is satisfied by definition. The condition (1)(a) can be proved by the Gabriel-Quillen embedding. We refer the reader to \cite[Proposition 2.9]{buhler} and \cite[Corollary 2.18]{buhler} for the conditions (1)(b) and (1)(c) respectively.
\end{proof}

The following proposition is a technical part of the proof of Theorem \ref{main}.
\begin{proposition}\label{keyprop}
Suppose that $\EE$ is idempotent complete and a class $F$ satisfies the conditions of Proposition \ref{bij} $(1)$. Put $\DD:=\DD(F)$. Then the following are equivalent.
\begin{enumerate}
\item $F$ is an exact structure on $\EE$.
\item $\DD$ is a Serre subcategory of $\mod\EE$ and $\Ext^2_\EE(\DD,\EE)$ is a Serre subcategory of $\mod\EE^{\op}$.
\end{enumerate}
\end{proposition}
\begin{proof}
(1) $\Rightarrow$ (2):
First we show that $\DD$ is a Serre subcategory of $\mod\EE$.
Let
\[
0 \to M_1 \to M \xrightarrow{k} M_2 \to 0
\]
be a short exact sequence in $\mod\EE$.

Assume that $M_1$ and $M_2$ are in $\DD$. We will see that $M \in \DD$. By the definition of $\DD$, there exists an $F$-exact sequence $X_i \xrightarrow{f_i} Y_i \xrightarrow{g_i} Z_i$ such that
\[
0 \to P_{X_i} \xrightarrow{P_{f_i}} P_{Y_i} \xrightarrow{P_{g_i}} P_{Z_i} \to M_i \to 0
\]
is exact for each $i=1,2$. By the horseshoe lemma, we obtain an exact commutative diagram in $\Mod\EE$ of the following form
\[
\xymatrix{
 & 0\ar[d] & 0\ar[d] & 0\ar[d] & 0\ar[d] & \\
0 \ar[r] & P_{X_1} \ar[r]^{P_{f_1}}\ar[d] & P_{Y_1} \ar[r]^{P_{g_1}}\ar[d] & P_{Z_1} \ar[r]\ar[d] & M_1 \ar[r]\ar[d] & 0 \\
0 \ar[r] & P_{X_1\oplus X_2} \ar[r]^{P_f}\ar[d] & P_{Y_1 \oplus Y_2} \ar[r]^{P_g}\ar[d] & P_{Z_1 \oplus Z_2} \ar[r]\ar[d] & M \ar[r]\ar[d]^{k} & 0 \\
0 \ar[r] & P_{X_2}\ar[d] \ar[r]^{P_{f_2}} & P_{Y_2} \ar[d]\ar[r]^{P_{g_2}} & P_{Z_2} \ar[d]\ar[r] & M_2 \ar[d]\ar[r] & 0 \\
 & 0 & 0 & 0 & 0 &
}
\]
where the columns except the right-most one are split exact.
Since $P_{(-)} : \EE \to \Mod\EE$ is fully faithful by the Yoneda Lemma, we obtain the following commutative diagram
\[
\xymatrix{
X_1 \inflr^{f_1} \infld & Y_1 \deflr^{g_1} \infld & Z_1 \infld \\
X_1 \oplus X_2 \ar[r]^{f}\defld & Y_1 \oplus Y_2 \ar[r]^{g} \defld & Z_1 \oplus Z_2 \defld \\
X_2 \inflr^{f_2} & Y_2 \deflr^{g_2} & Z_2
}
\]
in $\EE$. The top and bottom rows are $F$-exact, each of the three columns is split exact and $g f = 0$.
Thus we can apply $3 \times 3$ lemma in exact categories, see \cite[Corollary 3.6]{buhler}. Thus the middle row is also $F$-exact, which implies that $M$ is in $\DD$.

Next suppose that $M$ is in $\DD$. We first see that $M_1$ is in $\DD$. We have an $F$-exact sequence $X\xrightarrow{f} Y\xrightarrow{g} Z$ such that $0 \to P_X \xrightarrow{P_f} P_Y \xrightarrow{P_g} P_Z \xrightarrow{h} M \to 0$ is exact.
Since $M_1$ is in $\mod\EE$, we have a surjection $P_W \to M_1 \to 0$ for some $W$ in $\EE$. This yields the following commutative diagram
\[
\xymatrix{
& & & P_W \ar[r] \ar[d]^{P_\varphi}& M_1 \ar[r] \infld & 0 \\
0 \ar[r]& P_X \ar[r]^{P_f}& P_Y \ar[r]^{P_g}& P_Z \ar[r]^{h}& M \ar[r]& 0
}
\]
Since $F$ is an exact structure on $\EE$, there exists a pullback diagram
\begin{equation}\label{pb}
\xymatrix{
X \inflr^{a} \ar@{=}[d] & E \deflr^{b} \ar[d]^{c} & W \ar[d]^{\varphi} \\
X \inflr^{f} & Y \deflr^{g} & Z
}
\end{equation}
in $\EE$ where the top row is $F$-exact.
By using the universal property of pullbacks, one can easily check that
\[
0 \to P_X \xrightarrow{P_a} P_E \xrightarrow{P_b} P_W \to M_1 \to 0
\]
is exact. Thus $M_1$ is in $\DD$.

On the other hand, it is well-known that the complex $E \xrightarrow{^t [b, -c]} W\oplus Y \xrightarrow{[\varphi, g]} Z$ corresponding to the pullback square in (\ref{pb}) is $F$-exact (see e.g. \cite[Proposition 2.12]{buhler}). Moreover, one can easily check that
\[
P_W \oplus P_Y \xrightarrow{[P_{\varphi}, P_g]} P_Z \xrightarrow{k h} M_2 \to 0
\]
is exact, where $k h$ is the composition $h:P_Z \defl M$ and $k:M \defl M_2$. Therefore $N$ is in $\DD$, which completes the proof that $\DD$ is a Serre subcategory of $\mod\EE$.

Note that an object $M$ in $\mod\EE^{\op}$ is contained in $\Ext^2_\EE(\DD,\EE)$ if and only if there exists an $F$-exact sequence $X\xrightarrow{f} Y\xrightarrow{g} Z$ in $\EE$ such that $0 \to P^Z \xrightarrow{P^g} P^Y \xrightarrow{P^f} P^X \to M \to 0$ is exact. Hence $\Ext^2_\EE(\DD,\EE)$ is a Serre subcategory of $\mod\EE^{\op}$ by the dual argument.

(2) $\Rightarrow$ (1):
By duality, it suffices to show that $F$ satisfies (Ex0)-(Ex2).
Note that (Ex0) automatically holds by the conditions (1)(a), (c) and (d) in Proposition \ref{bij}.

(Ex1)
Let $A \xrightarrow{f} B \xrightarrow{g} C$ and $X \xrightarrow{h} C \xrightarrow{k} D$ be $F$-exact sequences. We show that $k g$ is also an $F$-epimorphism. We have the following commutative diagram with exact rows
\[
\xymatrix{
 & & & 0\ar[d] & & \\
 & & &P_X \ar@{=}[r]\ar[d]^{P_h} & P_X\ar[d] & \\
0 \ar[r]& P_A \ar[r]^{P_f}& P_B \ar@{=}[d] \ar[r]^{P_g}& P_C\ar[d]^{P_k} \ar[r]& L\ar[d] \ar[r]& 0\\
 & & P_B \ar[r]^{P_{kg}}& P_D\ar[d] \ar[r]& M\ar[d] \ar[r]& 0\\
 & & &N \ar@{=}[r]\ar[d] & N\ar[d] & \\
 & & &0 & 0 &
}
\]
where $L$ and $N$ are in $\DD$ by the definition of $\DD$. The right-most column is exact by diagram chasing. Since $\DD$ is a Serre subcategory of $\mod\EE$, we immediately have that $M$ is in $\DD$.
In particular, $M$ is contained in $\CC_2(\EE)$. By Lemma \ref{Schanuel}, there exists a kernel-cokernel pair $Y \xrightarrow{l} B \xrightarrow{kg} D$ such that $0 \to P_Y \xrightarrow{P_l} P_B \xrightarrow{P_{k g}} P_D \to M \to 0$ is exact. Thus $k g$ is an $F(\DD)$-epimorphism. Since $F(\DD)=F$, the morphism $k g$ is an $F$-epimorphism.

(Ex2)
Let $X \xrightarrow{f} Y \xrightarrow{g} Z$ be an $F$-exact sequence and $h:W \to Z$ an arbitrary morphism in $\EE$. Then we have the commutative diagram
\[
\xymatrix{
& & & & 0\ar[d] & \\
& & & P_W \ar[r] \ar[d]^{P_h} & L \ar[r] \ar[d] & 0 \\
0 \ar[r]& P_X \ar[r]^{P_f}& P_Y \ar[r]^{P_g} & P_Z \ar[r]^{a}& M \ar[r]\ar[d]^{b}& 0\\
& & & & N\ar[d] & \\
& & & & 0 &
}
\]
where all the rows and columns are exact.
Since $M$ is in $\DD$ by the definition, $L$ and $N$ are also in $\DD$. In particular, $N$ is contained in $\CC_2(\EE)$.
On the other hand, we have that $P_W \oplus P_Y \xrightarrow{[P_h, P_g]} P_Z \xrightarrow{b a} N \to 0$ is exact. Thus by Lemma \ref{Schanuel}, there exists an exact sequence
\[
\xymatrix{
0 \ar[r]& P_E
\ar[r]^(.4){\begin{sbmatrix} P_k \\-P_l\end{sbmatrix}}
 & P_W \oplus P_Y \ar[r]^(.6){[ P_h, P_g]} & P_Z \ar[r]& N \ar[r]^{b a}& 0
 }
\]
in $\Mod\EE$. It is standard that this exact sequence yields a pullback diagram
\begin{equation}\label{pb2}
\xymatrix{
E \ar[d]^{l}\ar[r]^{k} & W \ar[d]^{h} \\
Y \ar[r]^{g} & Z
}
\end{equation}
in $\EE$. Thus the existence of the pullback has been proved.
By the universal property of the pullback (\ref{pb2}), there exists a complex $X \xrightarrow{i} E \xrightarrow{k} W$ in $\EE$ such that
\[
0 \to P_X \xrightarrow{P_i} P_E \xrightarrow{P_k} P_W \to L \to 0
\]
is exact. Thus the complex $X \xrightarrow{i} E \xrightarrow{k} W$ belongs to $F(\DD) = F$, which implies that $k$ is an $F$-epimorphism as desired.
\end{proof}

\begin{proof}[Proof of Theorem \ref{main}]
It is immediate from Proposition \ref{keyprop} and Lemma \ref{ok}.
\end{proof}

The following description of projective objects in $\EE$ in terms of $\DD$ is needed later.
\begin{proposition}\label{proj}
Let $\EE$ be an idempotent complete exact category and $\DD$ the corresponding Serre subcategory of $\mod\EE$ given in Theorem \ref{main}. Then for an object $Z$ in $\EE$, the following are equivalent.
\begin{enumerate}
\item $W$ is projective in $\EE$.
\item $M(W)=0$ for all object $M \in \DD$.
\end{enumerate}
\end{proposition}
\begin{proof}
The category $\DD$ consists of all functors $M$ such that $P_Y \xrightarrow{P_g} P_Z \to M \to 0$ is exact for some deflation $g:Y \defl Z$, so (2) is equivalent to that $P_Y(W) \xrightarrow{P_g(W)} P_Z(W)$ is surjective for every deflation $g$. This occurs if and only if $W$ is a projective object in $\EE$.
\end{proof}

Next let us consider the particular case when the simpler classification is available. This includes the case when $\EE$ is abelian, or more generally, quasi-abelian (see \cite{rump} for the detail).

\begin{lemma}\label{tokubetsu}
Let $\EE$ be an idempotent complete category. Then the following are equivalent.
\begin{enumerate}
\item $\CC_2(\EE)$ and $\CC_2(\EE^{\op})$ are Serre subcategories of $\mod \EE$ and $\mod \EE^{\op}$ respectively.
\item The class of all kernel-cokernel pairs in $\EE$ defines an exact structure of $\EE$.
\end{enumerate}
In this case, there exists a bijection between the following two classes.
\begin{enumerate}
\item Exact structures on $\EE$.
\item Serre subcategories of $\CC_2(\EE)$.
\end{enumerate}
\end{lemma}
\begin{proof}
Let $F$ be the class of all kernel-cokernel pairs in $\EE$. Then $\DD(F) = \CC_2(\EE)$ holds, thus Proposition \ref{keyprop} applies. The latter part is clear from Theorem \ref{main}.
\end{proof}

Now we will show that a more familiar description for $\DD$ is available if $\EE$ has enough projectives. First we recall the notation of the projectively stable category of an exact category. Let $\EE$ be an exact category with enough projective objects.
Denote by $[\PP](X,Y)$ the set of all morphisms from $X$ to $Y$ which factor through projective objects in $\EE$. Then $[\PP]$ is a two-sided ideal of $\EE$ and we denote by $\un{\EE} := \EE/[\PP]$, which we call the \emph{projectively stable category}.

\begin{lemma}\label{defect}
Let $\EE$ be an exact category with enough projectives and $\DD$ the subcategory of $\mod\EE$ corresponding to all conflations (see Definition \ref{maps}). Then $\DD \equi \mod_1 \un{\EE}$ holds.
\end{lemma}
\begin{proof}
We have an embedding $\mod_1 \un{\EE} \to \mod \EE$ and denote its essential image by $\DD'$. We show that the following conditions are equivalent for an object $M$ in $\mod\EE$.
\begin{enumerate}
\item $M \in \DD$.
\item There exist a deflation $g: Y \to Z$ in $\EE$ and an exact sequence
\begin{equation}\label{ei}
\un{\EE}(-,Y) \xrightarrow{(-)\circ g} \un{\EE}(-,Z) \to M \to 0.
\end{equation}
\item $M \in \DD'$.
\end{enumerate}

(1) $\Leftrightarrow$ (2):
This is easily shown by diagram chasing.

(2) $\Rightarrow$ (3):
Clear.

(3) $\Rightarrow$ (2):
Suppose that $M$ is in $\DD'$. Since $M$ is finitely presented $\un{\EE}$-module, we have an exact sequence of the form (\ref{ei}) for some $g: Y \to Z$. By assumption, there exists a deflation $\varphi: P \defl Z$ for some projective object $P$. Using this, we may replace $g$ by $[g,\varphi]:Y \oplus P \to Z$, which is a deflation. This proves the claim.
\end{proof}

Combining Lemma \ref{tokubetsu} and \ref{defect}, we immediately obtain the following conclusion, which gives a generalization of Buan's result \cite[Proposition 3.3.2]{bu}, where $\EE$ was assumed to be $\mod\Lambda$ for some artin algebra $\Lambda$.
\begin{corollary}
Let $\EE$ be an idempotent complete additive category such that the class of all kernel-cokernel pairs defines an exact structure with enough projectives (e.g. abelian category with enough projectives). Denote by $\un{\EE}$ the projectively stable category in this exact structure. Then there exists a bijection between the following two classes.
\begin{enumerate}
\item Exact structures on $\EE$.
\item Serre subcategories of $\mod_1 \un{\EE}$.
\end{enumerate}
\end{corollary}

\section{Exact categories of finite type}
We use our previous results to classify idempotent complete exact categories of finite type.
For an additive category $\EE$, an object $M$ is called an \emph{additive generator} of $\EE$ if $\add M = \EE$ holds, where $\add M$ is the subcategory of $\EE$ consisting of all direct summands of finite direct sums of $M$. We call that an additive category $\EE$ is \emph{of finite type} if it has an additive generator.

We say that an additive category $\EE$ is a \emph{Krull-Schmidt category} if every object in $\EE$ is a finite direct sum of indecomposable objects whose endomorphism rings are local. For the basics of Krull-Schmidt categories, we refer the reader to \cite{krause}. An additive category $\EE$ is Krull-Schmidt if and only if $\EE$ is idempotent complete and $\End_\EE(X)$ is semiperfect for every $X \in \EE$.
If $\EE$ is Krull-Schmidt, then $\EE$ is of finite type precisely when $\EE$ has finitely many indecomposables up to isomorphism.

Let $M$ be an additive generator of $\EE$ and $\Gamma := \End_\EE(M)$. Then we have a fully faithful functor $\EE(M,-): \EE \to \Mod \Gamma$ to the category of right $\Gamma$-modules. Its essential image coincides with the category $\proj\Gamma$ of finitely generated projective $\Gamma$-modules precisely when $\EE$ is idempotent complete. Thus, when we deal with an idempotent complete additive category of finite type, we may assume that $\EE = \proj\Gamma$ for some ring $\Gamma$. To avoid technical complications, \emph{we restrict our attention to the case when $\Gamma$ is a noetherian ring}. Note that $\proj \Gamma$ is Krull-Schmidt if and only if $\Gamma$ is semiperfect.

\subsection{Basic properties}
We start with reformulating our result in Section 3 in terms of the ring $\Gamma$, where we use the same notation $\CC_2(\Gamma)$:
\[
\CC_2(\Gamma) := \{ M \in \mod\Gamma \text{ $|$ } \pd M_\Gamma \leq 2 \text{ and } \Ext_\Gamma^i(M,\Gamma) = 0 \text{ for } i =0,1 \}.
\]
Note that every non-zero module $M$ in $\CC_2(\Gamma)$ satisfies $\pd M_\Gamma = 2$.

\begin{theorem}\label{finitecorresp}
Let $\Gamma$ be a noetherian ring and put $\EE := \proj\Gamma$. Then there exists a bijection between the following two classes.
\begin{enumerate}
\item Exact structures $F$ on $\EE$.
\item Subcategories $\DD$ of $\CC_2(\Gamma)$ satisfying the following condition.
 \begin{enumerate}
  \item $\DD$ is a Serre subcategory of $\mod\Gamma$.
  \item $\Ext^2_\Gamma(\DD,\Gamma)$ is a Serre subcategory of $\mod\Gamma^{\op}$.
 \end{enumerate}
\end{enumerate}
The correspondence is given as follows.
 \begin{enumerate}
  \item[$\bullet$] For a given $\DD$, a complex $X \xrightarrow{f} Y \xrightarrow{g} Z$ in $\EE$ is in $F$ if and only if $0 \to X \xrightarrow{f} Y \xrightarrow{g} Z \to M \to 0$ is an exact sequence in $\mod\Gamma$ with some $M$ in $\DD$.
  \item[$\bullet$] For a given $F$, a $\Gamma$-module $M \in \mod\Gamma$ is in $\DD$ if and only if there exists a deflation $g$ in $\EE$ with $M \iso \coker g$.
 \end{enumerate}
\end{theorem}

When we deal with exact structures on $\proj\Gamma$, the $2$-regular condition for simple modules (see Definition \ref{2regdef}) are quite essential.
For the set $\SS$ of simple $\Gamma$-modules, we denote by $\Filt\SS$ the subcategory of $\mod\Gamma$ consisting of all modules $M$ such that $M$ has finite length and all composition factors of $M$ are contained in $\SS$. The following observation is useful.

\begin{lemma}\label{simplemma}
Let $\Gamma$ be a noetherian ring and $\SS$ a set of simple $\Gamma$-modules. Then the following are equivalent.
\begin{enumerate}
\item Every module in $\SS$ satisfies the $2$-regular condition.
\item $\DD := \Filt \SS$ is contained in $\CC_2(\Gamma)$ and satisfies the conditions of Theorem \ref{finitecorresp}(2).
\item There exists a subcategory $\DD$ of $\CC_2(\Gamma)$ containing $\SS$ such that $\DD$ satisfies the conditions of Theorem \ref{finitecorresp}(2).
\end{enumerate}
\end{lemma}
\begin{proof}
(1) $\Rightarrow$ (2):
Clearly $S \in \CC_2(\Gamma)$ holds for every $S \in \SS$. Since $\CC_2(\Gamma)$ is closed under extensions in $\mod\Gamma$, we have that $\Filt S$ is a Serre subcategory of $\mod\Gamma$ contained in  $\CC_2(\Gamma)$. Since $\Ext^2_\Gamma(-,\Gamma)$ gives an exact duality $\CC_2(\Gamma) \equi \CC_2(\Gamma ^{\op})$ and $\Ext^2_\Gamma(S,\Gamma)$ is simple for every $S \in \SS$, we have $\Ext^2_\Gamma(\DD,\Gamma) = \Filt\Ext^2_\Gamma(\SS,\Gamma)$. Thus this category is a Serre subcategory of $\mod\Gamma^{\op}$, which shows that $\Filt \SS$ satisfies the conditions of Theorem \ref{finitecorresp}(2).

(2) $\Rightarrow$ (3): Obvious.

(3) $\Rightarrow$ (1):
Let $S$ be a simple module contained in $\SS$. It follows from $S \in  \CC_2(\Gamma)$ that $\pd S_\Gamma = 2$ and $\Ext^i_\Gamma(S,\Gamma)=0$ for $i=0,1$. Notice that $\Ext^2_\Gamma(-,\Gamma)$ gives the duality between $\DD$ and $\Ext^2_\Gamma(\DD,\Gamma)$, both of which are abelian categories. Thus it is immediate that $\Ext^2_\Gamma(S,\Gamma)$ is a simple object in the abelian category $\Ext^2_\Gamma(\DD,\Gamma)$, which implies that $\Ext^2_\Gamma(S,\Gamma)$ is a simple left $\Gamma$-module. Therefore $S$ satisfies the $2$-regular condition.
\end{proof}

Every additive category admits a trivial exact structure whose conflations are split exact sequences. We have the following criterion on the existence of non-trivial exact structures on $\proj\Gamma$.
\begin{proposition}\label{criterion}
Let $\Gamma$ be a noetherian ring. Then $\proj\Gamma$ admits a non-trivial exact structure if and only if there exists a simple $\Gamma$-module satisfying the $2$-regular condition.
\end{proposition}
\begin{proof}
Suppose that there exists a simple $\Gamma$-module $S$ satisfying the $2$-regular condition. Then Lemma \ref{simplemma} and Theorem \ref{finitecorresp} imply the existence of a non-trivial exact structure on $\proj\Gamma$.

Conversely, suppose that $\proj\Gamma$ has a non-trivial exact structure. By Theorem \ref{finitecorresp}, we have a non-zero Serre subcategory $\DD$ of $\mod\Gamma$. Since any non-zero $\Gamma$-module in $\mod\Gamma$ has a surjection onto simple $\Gamma$-module, $\DD$ contains at least one simple $\Gamma$-module $S$. Then Lemma \ref{simplemma} implies that $S$ satisfies the $2$-regular condition.
\end{proof}

\begin{example}
Let $R$ be a commutative noetherian local ring. Then Proposition \ref{criterion} implies that there exists a non-trivial exact structure on $\proj R$ if and only if $R$ is a regular local ring of dimension $2$. In this case, $\proj R$ has exactly one non-trivial exact structures. In this exact structure, $P_2 \xrightarrow{f} P_1 \xrightarrow{g} P_0$ is a conflation if and only if $f$ is a kernel of $g$ and $\coker(g)$ is of finite length over $R$.
\end{example}

\subsection{Admissible exact structures}
We introduce a nice class of exact categories which is completely controlled by the simple modules satisfying the $2$-regular condition. For a ring $\Gamma$, we denote by $\fl \Gamma$ the subcategory of $\mod\Gamma$ consisting of $\Gamma$-modules of finite length. Similarly, we denote by $\fl \EE$ the category consisting of $\EE$-modules of finite length.
\begin{definition}\label{admdef}
Let $(\EE, F)$ be an exact category and $\DD$ the subcategory of $\mod\EE$ corresponding to $F$ under Theorem \ref{main}. We say that $F$ is admissible if $\DD \subset \fl \EE$ holds.
\end{definition}
Let $(\proj \Gamma, F)$ be an exact category for a noetherian ring $\Gamma$ and $\DD$ the subcategory of $\mod\Gamma$ corresponding to $F$ under Theorem \ref{finitecorresp}. Then it is clear that $F$ is admissible if and only if $\DD \subset \fl \Gamma$ holds. Therefore, for an artinian ring $\Gamma$, every exact structure on $\proj\Gamma$ is admissible.

We show that the admissibility is left-right symmetric under some assumptions.
\begin{proposition}
Let $(\proj \Gamma, F)$ be an exact category for a noetherian ring $\Gamma$. Then it is admissible if and only if $(\proj \Gamma^{\op},F^{\op})$ is admissible.
\end{proposition}
\begin{proof}
Let $\DD$ be the subcategory of $\mod\Gamma$ corresponding to $F$ under Theorem \ref{finitecorresp}. Then $\Ext^2_\Gamma(\DD,\Gamma)$ is the subcategory corresponding to $F^{\op}$. Since $F$ is admissible, every object in $\DD$ has finite length as a $\Gamma$-module. Thus $\DD$ is an abelian category in which every object has finite length. Because $\Ext^2_\Gamma(\DD,\Gamma)$ is dual to $\DD$, every object in $\Ext^2_\EE(\DD,\EE)$ has finite length in the abelian category $\Ext^2_\Gamma(\DD,\Gamma)$. Since $\Ext^2_\Gamma(\DD,\Gamma)$ is a Serre subcategory of $\mod \Gamma^{\op}$, we have $\Ext^2_\Gamma(\DD,\Gamma) \subset \fl \Gamma^{\op}$.
\end{proof}

Using this notion, we can classify all admissible exact structures on $\proj\Gamma$.
\begin{theorem}\label{2regularcorresp}
Let $\Gamma$ be a noetherian ring. For $\EE := \proj\Gamma$, there exists a bijection between the following two classes.
\begin{enumerate}
\item Admissible exact structures $F$ on $\EE$.
\item Sets $\SS$ of isomorphism classes of simple $\Gamma$-modules satisfying the $2$-regular condition.
\end{enumerate}
The correspondence is given as follows.
 \begin{enumerate}[label=$\bullet$]
  \item For a given $\SS$, a complex $X \xrightarrow{f} Y \xrightarrow{g} Z$ in $\EE$ is a conflation if and only if there exists an exact sequence $0 \to X \xrightarrow{f} Y \xrightarrow{g} Z \to M \to 0$ in $\mod\Gamma$ with $M$ in $\Filt\SS$.
  \item For a given $F$, a simple $\Gamma$-module $S$ is in $\SS$ if and only if there exists a deflation $g:Y \to Z$ in $\EE$ such that $\coker(g) \iso S$ in $\mod\Gamma$.
 \end{enumerate}
\end{theorem}
\begin{proof}
By Theorem \ref{finitecorresp}, there exists a bijection between (1) and
\begin{enumerate}
\item[(3)] Subcategories $\DD$ of $\CC_2(\Gamma)$ satisfying the following condition.
 \begin{enumerate}
  \item $\DD$ is a Serre subcategory of $\mod\Gamma$ satisfying $\DD \subset \fl \Gamma$.
  \item $\Ext^2_\Gamma(\DD,\Gamma)$ is a Serre subcategory of $\mod\Gamma^{\op}$.
 \end{enumerate}
\end{enumerate}

We have mutually inverse bijections between (2) and (3) as follows;
For a given $\SS$ in (2), we put $\DD :=\Filt \SS$, and
for a given $\DD$ in (3), we denote by $\SS$ the set of simple modules contained in $\DD$.
By Lemma \ref{simplemma}, these maps are well-defined. These are mutually inverse to each other by definition.
\end{proof}

Next we will focus on the case over the ground ring $R$. \emph{In the rest of this subsection, we fix a commutative noetherian complete local ring $R$}.
For an additive $R$-category $\EE$ for a commutative noetherian ring $R$, we say that $\EE$ is \emph{Hom-noetherian} (resp. \emph{Hom-finite}) if the $R$-module $\EE(X,Y)$ is finitely generated (resp. of finite length) for every $X, Y \in \EE$. Every Hom-noetherian idempotent complete $R$-category is Krull-Schmidt.
For an exact $R$-category $\EE$, we say that $\EE$ is \emph{Ext-noetherian} (resp. \emph{Ext-finite}) if for every object $X$ and $Y$, the $R$-module $\Ext^1_\EE(X,Y)$ is finitely generated (resp. of finite length).

The following figure illustrates the relationship between these concepts.
We will prove these implication in Proposition \ref{admissible} and Corollary \ref{projinj}. We refer the reader to Appendix A.2 for the results and notions we need in the proof.
\begin{figure}[h]
\begin{tikzpicture}
 \node (0) at (-3,0)  {$R$ is artinian};
 \node [text width=2cm, text centered] (1) at (0,0)  {enough proj. (or inj.)};
 \node (2) at (3,0)  {Ext-noeth.};
 \node (3) at (6,0) {Ext-finite};
 \node (4) at (9,0) {admissible};

 \draw [-implies,double equal sign distance] (0) -- (1) ;
 \draw [-implies,double equal sign distance] (1) -- (2) ;
 \draw [-implies,double equal sign distance] (2) to [bend left=10] node[auto] {\footnotesize finite type} (3) ;
 \draw [-implies,double equal sign distance] (3) to [bend left=10]  (2) ;
 \draw [-implies,double equal sign distance] (3) to [bend left=10] node[auto] {\footnotesize finite type} (4) ;
 \draw [-implies,double equal sign distance] (4) to [bend left=10] node[auto] {\footnotesize enough inj.} (3) ;

\end{tikzpicture}
\caption{Some implications on admissibility}
\label{admfigure}
\end{figure}

\begin{proposition}\label{admissible}
Let $\EE$ be a Hom-noetherian idempotent complete exact $R$-category. Then the following hold.
\begin{enumerate}
\item If $\EE$ has either enough projectives or enough injectives, then $\EE$ is Ext-noetherian.
\item If $\EE$ is Ext-noetherian and $\EE$ is of finite type, then $\EE$ is Ext-finite.
\item If $\EE$ is Ext-finite and $\EE$ is of finite type, then $\EE$ is admissible. Conversely, if $\EE$ is admissible and has enough injectives, then $\EE$ is Ext-finite.
\end{enumerate}
\end{proposition}

\begin{proof}
(1)
This is clear since extension groups can be computed by projective resolutions or injective coresolutions.

(2)
If $\EE$ is of finite type, then $\EE$ has AR conflations by Corollary \ref{exist}. Since $\EE$ is Ext-noetherian, Proposition \ref{ardual} implies that $\EE$ is Ext-finite.

(3)
Suppose that $\EE$ is Ext-finite and of finite type. Take $W$ to be an additive generator of $\EE$ and $\Gamma := \End_\EE W$. For any conflation $X \xrightarrow{f} Y \xrightarrow{g} Z$, the cokernel of $g \circ (-):\EE(W,Y) \to \EE(W,Z)$ is a submodule of $\Ext_\EE^1(W,X)$, thus is of finite length over $R$, or equivalently, over $\Gamma$. Therefore $\EE$ is admissible.

Conversely, suppose that the exact structure on $\EE$ is admissible and $\EE$ has enough injectives. For an object $X \in \EE$, take a conflation $X \xrightarrow{f} I \xrightarrow{g} Z$ such that $I$ is injective. Then $\EE(-,I) \xrightarrow{g \circ (-)} \EE(-,Z) \to \Ext^1_\EE(-,X) \to 0$ is exact, hence $\Ext^1_\EE(-,X)$ is of finite length. It follows that $\Ext^1_\EE(W,X)$ is of finite length over $R$ for any $W$ in $\EE$.
\end{proof}

Next we interpret Theorem \ref{2regularcorresp} in terms of the quiver of $\Gamma$. Recall that for a noetherian $R$-algebra $\Gamma$, the valued quiver $Q(\Gamma)$ is defined as follows, where $\JJ$ denotes the radical of $\proj\Gamma$.
\begin{enumerate}
\item The set of vertices is $\ind (\proj\Gamma)$, that is, the isomorphism classes of all indecomposable projective right $\Gamma$-modules.
\item We draw an arrow from $P$ to $Q$ if $\JJ(P,Q)/ \JJ^2(P,Q) \neq 0$ with a valuation $(d_{P,Q}, d'_{P,Q})$, where $d_{P,Q}$ (resp. $d'_{P,Q}$) is the dimension of $\JJ(P,Q)/ \JJ^2(P,Q)$ as a $k_P$-vector space (resp. $k_Q$-vector space). Here $k_P := \End_\EE(P)/ \rad \End_\EE(P)$ and $k_Q := \End_\EE(Q)/ \rad \End_\EE(Q)$.
\end{enumerate}

We introduce a \emph{translation} on this quiver $Q(\Gamma)$.
\begin{definition}\label{const}
Let $\Gamma$ be a noetherian $R$-algebra. The translation $\tau$ of $P \in Q(\Gamma)$ is defined when $P/\rad P$ satisfies the $2$-regular condition. In this case, $\tau P$ is the projective module $Q$ such that $\Hom_\Gamma(Q,\Gamma)$ is a projective cover of the simple left $\Gamma$-module $\Ext^2_\Gamma(P/\rad P,\Gamma)$.
We draw a \emph{dotted arrow} from $P$ to $\tau P$ whenever $\tau P$ is defined.
\end{definition}
This construction yields a valued translation quiver (see \cite{ars}). Since we will not use valued arrows, we omit the proof.

Now admissible exact structures on $\proj\Gamma$ can be visually classified by the dotted arrows.
\begin{corollary}\label{dotted}
Let $\Gamma$ be a noetherian $R$-algebra over a noetherian complete local ring $R$. Then there exists a bijection between the following two classes.
\begin{enumerate}
\item Admissible exact structures on $\proj \Gamma$.
\item Sets of dotted arrows in $Q(\Gamma)$.
\end{enumerate}
Moreover, the Auslander-Reiten quiver of the exact category $\EE$ is given by the quiver $Q(\Gamma)$ with the dotted arrows chosen in {\upshape (2)}.
\end{corollary}
\begin{proof}
Note that each dotted arrow $\tau P \dashleftarrow P$ in $Q(\Gamma)$ bijectively corresponds to the simple $\Gamma$-module $P / \rad P$ satisfying the $2$-regular condition by the definition. Thus Theorem \ref{2regularcorresp} implies the assertion.
\end{proof}

In the situation of Corollary \ref{dotted}, let $\XXX$ be a set of dotted arrows in $Q(\Gamma)$ and $F$ be the corresponding exact structures on $\EE:=\proj\Gamma$. For the convenience of the reader, we give several relations between $\XXX$ and $F$.
\begin{itemize}
  \item $\XXX$ is empty if and only if $F$ is the split exact structure on $\EE$, that is, the smallest exact structure in which only split exact sequences are conflations. This is nothing but the ``usual" exact structure on $\proj\Gamma$, which is induced from the embedding $\proj\Gamma \subset \mod\Gamma$.
  \item $\XXX$ is the set of all dotted arrows if and only if $F$ is the unique maximal exact structures among admissible ones on $\EE$. In particular, if $R$ is artinian, then this holds precisely when $F$ is the unique maximal exact structure, because all the exact structures are admissible (see Figure \ref{admfigure}).
  \item Let $W$ be an indecomposable object in $\EE$. Then $W$ is a projective (resp. injective) object in $(\EE,F)$ if and only if $W$ is not the source (resp. not the target) of a dotted arrow in $\XXX$. This follows from the description of projective objects given in the proof of Proposition \ref{enough}(1) below.
\end{itemize}

We end this subsection by giving examples of Corollary \ref{dotted}.
\begin{example}
Let $\Gamma$ be the algebra defined in Example \ref{example1} the introduction. Then $Q(\Gamma)$ coincides with Figure \ref{cap1} (all valuations are trivial, i.e. $(1,1)$ and all dotted lines are interpreted as arrows from right to left). It has seven dotted arrows, thus $\proj \Gamma$ has $2^7 = 128$ exact structures.
Actually, this example is due to \cite[Example 5.7]{en}, and in that paper it is shown that the category $\proj\Gamma$ is a \emph{strict $\tau$-category}, introduced by Iyama (see e.g. \cite{tau4}). For such a class of categories, we have a nice correspondence between translation quivers and their mesh algebras: the quiver $Q(\Gamma)$ of the mesh algebra $\Gamma$ coincides with the original translation quiver.
\end{example}
The following two examples concern with the category $\CM\Lambda$ of Cohen-Macaulay $\Lambda$-modules over an order $\Lambda$. We refer the reader to Section 4 for the details on orders. The usual exact structure on the category $\CM\Lambda$ is admissible if $\CM\Lambda$ is of finite type, since $\CM\Lambda$ has enough projectives (Proposition \ref{admissible}). Therefore, the usual exact structure on $\CM\Lambda$ corresponds to the set of dotted arrows of the usual Auslander-Reiten quiver of $\CM\Lambda$.
\begin{example}
Let $k$ be a field, $R:=k\llbracket t\rrbracket$ the ring of formal power series over a field $k$ and
\[
\Gamma :=
\begin{bmatrix}
R & R & \cdots & R & R \\
(t) & R & \cdots & R & R \\
(t^2) & (t) & \ddots & R & R \\
\vdots & \ddots & \ddots & \ddots & \vdots \\
(t^n) & \cdots & (t^2) & (t) & R
\end{bmatrix}.
\]
Then $Q(\Gamma)$ is given by the following quiver.
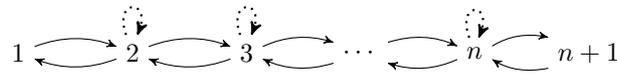
\[
\begin{tikzpicture}[xscale=1.5]
\node (1) at (0,0) {$1$};
\node (2) at (1,0) {$2$};
\node (3) at (2,0) {$3$};
\node (4) at (3,0) {$\cdots$};
\node (5) at (4,0) {$n$};
\node (6) at (5,0) {$n+1$};

\draw[->] (1) to [bend left] (2);
\draw[->] (2) to [bend left](1);
\draw[->] (2) to [bend left](3);
\draw[->] (3) to [bend left](2);
\draw[->] (3) to [bend left](4);
\draw[->] (4) to [bend left](3);
\draw[->] (4) to [bend left](5);
\draw[->] (5) to [bend left](4);
\draw[->] (5) to [bend left](6);
\draw[->] (6) to [bend left](5);

\draw[thick, dotted,->] (2) to [loop above](2);
\draw[thick, dotted,->] (3) to [loop above](3);
\draw[thick, dotted,->] (5) to [loop above](5);

\end{tikzpicture}
\]
Note that $\Gamma$ is the Auslander order for an $R$-order
\[
\Lambda :=
\begin{bmatrix}
R & R\\
(t^n) & R
\end{bmatrix}
\]
hence there is an equivalence $\proj \Gamma \equi \CM \Lambda$, and the above translation quiver gives the Auslander-Reiten quiver of $\CM \Lambda$.
Corollary \ref{dotted} shows that $\proj\Gamma$ has $2^{n-1}$ admissible exact structures, and the usual exact structure on $\CM \Lambda$ corresponds to the set of all dotted arrows.
In general, for an Cohen-Macaulay-finite $R$-order $\Lambda$ over a complete discrete valuation ring $R$, the quiver $Q(\Gamma)$ of its Auslander order $\Gamma$ coincides with the usual Auslander-Reiten quiver of $\CM\Lambda$.
\end{example}

\begin{example}\label{veronese}
Let $T= k \llbracket x,y \rrbracket$ be the ring of formal power series in two variables over a field $k$ of characteristic $0$, and let $R = T^{(n)}$ be the $n$-th Veronese subring of $T$, that is, $R = k \llbracket x^n,x^{n-1}y, \cdots, y^n \rrbracket$.
 Put $\Gamma := \End_R(T)$. Then all simple $\Gamma$-modules satisfy the $2$-regular condition, and the translation quiver $Q(\Gamma)$ is the following.
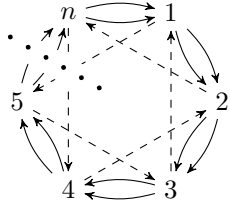
\[
\begin{tikzpicture}[scale=0.5]
   \newdimen\R
   \R=2.7cm
   \node (1) at (60:\R)  {$1$};
   \node (2) at (0:\R)  {$2$};
   \node (3) at (-60:\R)  {$3$};
   \node (4) at (-120:\R)  {$4$};
   \node (5) at (180:\R)  {$5$};
   \node (n-1) at (120:\R)  {$n$};

\draw[->] (1) to [bend left=15](2);
\draw[->] (2) to [bend left=15](3);
\draw[->] (3) to [bend left=15](4);
\draw[->] (4) to [bend left=15](5);
\draw[->] (5) to [bend left=15](n-1);
\draw[->] (n-1) to [bend left=15](1);
\draw[->] (1) to [bend right=15](2);
\draw[->] (2) to [bend right=15](3);
\draw[->] (3) to [bend right=15](4);
\draw[->] (4) to [bend right=15](5);
\draw[->] (5) to [bend right=15](n-1);
\draw[->] (n-1) to [bend right=15](1);

\draw[dashed,->] (3) to (1);
\draw[dashed,->] (4) to (2);
\draw[dashed,->] (5) to (3);
\draw[dashed,->] (2) to (n-1);
\draw[dashed,->] (1) to (5);
\draw[dashed,->] (n-1) to (4);

\node [draw=white,fill=white,rotate=-30,inner ysep=1pt] at (150:\R*0.7) {\huge $\cdots\cdots$};
\end{tikzpicture}
\]
Here the double arrows are interpreted as the single arrows with valuation $(2,2)$. Note that $\add (T_R)= \CM R$ holds, thus $\proj\Gamma \equi \CM R$ (see \cite[Chapter 10]{yo} for the detail). Therefore Corollary \ref{dotted} implies that $\CM R$ has $2^n$ admissible exact structures.
The usual exact structure on $\CM R$ corresponds to the set of all but one arrows above, which is the usual Auslander-Reiten quiver of $\CM R$.
 On the other hand, the exact structure on $\proj \Gamma$ corresponding to all dotted arrows arises from the embedding $\proj\Gamma \to \mod\Gamma \defl \mod\Gamma / \fl \Gamma$, where $\mod\Gamma / \fl \Gamma$ is the localization of $\mod\Gamma$ by the Serre subcategory $\fl\Gamma$ (we omit the details here).
\end{example}

\subsection{Enough projectivity and admissibility}
We collect some properties about the relation between enough projectivity and admissibility of exact $R$-categories.
\emph{Throughout this subsection, we fix a commutative noetherian complete local ring $R$}.

By Proposition \ref{admissible}(3), If $\EE$ is a Hom-noetherian idempotent complete exact $R$-category of finite type which has enough projectives, then $\EE$ is admissible. The following gives a criterion when the converse holds.

\begin{proposition}\label{enough}
Let $\EE:=(\proj\Gamma,F)$ be an admissible exact category for a semiperfect noetherian ring $\Gamma$. Take an idempotent $e \in \Gamma$ such that $e \Gamma$ is an additive generator of projective objects in $\EE$. Then the following hold.
\begin{enumerate}
\item $\DD = \fl (\Gamma/\Gamma e \Gamma)$ holds, where $\DD$ is the subcategory of $\mod\Gamma$ corresponding to $F$ under Theorem \ref{finitecorresp}.
\item $\EE$ has enough projectives if and only if $\Gamma / \Gamma e \Gamma$ is in $\fl \Gamma$.
\end{enumerate}
\end{proposition}
\begin{proof}
(1)
Since $\Gamma$ is semiperfect, $\EE$ is a Krull-Schmidt exact category. By Proposition \ref{proj}, an object $P$ in $\EE$ is projective if and only if $\Hom_\Gamma(P,\DD) =0$. Thus an indecomposable object in $\EE$ is projective in $\EE$ if and only if it is the projective cover of a simple $\Gamma$-module which is not contained in $\DD$.
Therefore, a simple $\Gamma$-module $S$ is contained in $\DD$ if and only if $\Hom_\EE (e \Gamma,S) = 0$, that is, $S e = 0$. Thus, for a $\Gamma$-module $M \in \mod\Gamma$, it follows that $M \in\DD$ holds if and only if $M e = 0$ and $M \in \fl \Gamma$, which implies the assertion.

(2)
Recall that a morphism $g:P \to X$ in $\EE$ is a deflation if and only if $\coker(g)$ in $\mod\Gamma$ is in $\DD$, and every projective resolution of $M \in \DD$ yields a conflation in $\EE$. Thus $\EE$ has enough projectives if and only if there exists an exact sequence
\begin{equation}\label{enougheq}
P \xrightarrow{g} \Gamma \to M \to 0
\end{equation}
 in $\mod\Gamma$ for some objects $P \in \add(e\Gamma)$ and $M \in \DD$.
Suppose that this holds. Then $\im(g) \subset \Gamma e \Gamma$ holds, so we have a surjection $M \defl \Gamma / \Gamma e \Gamma$. Since $M$ is in $\DD$, it follows that $M$ is of finite length, thus so is $\Gamma / \Gamma e \Gamma$.
Conversely, suppose that $\Gamma / \Gamma e \Gamma$ has finite length. Then $\Gamma / \Gamma e \Gamma$ is contained in $\DD = \fl (\Gamma / \Gamma e \Gamma)$. Since $\Gamma$ is noetherian, $\Gamma e \Gamma$ is a finitely generated as a right $\Gamma$-module. Therefore there exists an exact sequence of the form (\ref{enougheq}) with $M = \Gamma / \Gamma e \Gamma$.
\end{proof}
Consequently, we have the following interesting consequences. It is remarkable that we use purely module-theoretical argument to show non-trivial properties of exact categories.
\begin{corollary}\label{projinj}
Let $\EE$ be a Hom-noetherian idempotent complete admissible exact $R$-category of finite type. Then the following holds.
\begin{enumerate}
\item If $R$ is artinian, then $\EE$ has enough projectives and injectives.
\item $\EE$ has enough projectives if and only if $\EE$ has enough injectives.
\end{enumerate}
\end{corollary}
\begin{proof}
(1) Immediate from Proposition \ref{enough}(2).

(2)
It suffices to show the ``only if'' part.
We may assume $\EE = (\proj \Gamma,F)$ for a noetherian $R$-algebra $\Gamma$. Denote by $\DD = \Filt \SS$ the Serre subcategory of $\mod\Gamma$ which corresponds to $F$ under Theorem \ref{2regularcorresp}, and take idempotents $e$ and $f$ in $\Gamma$ such that $e \Gamma$ (resp. $f \Gamma$) is an additive generator of projective objects (resp. injective objects) in $\EE$. Put $\un{\Gamma} := \Gamma / \Gamma e \Gamma$ and $\ov{\Gamma}:= \Gamma / \Gamma f \Gamma$. We know from Proposition \ref{enough}(2) that $\un{\Gamma}$ is of finite length, and it suffices to show that $\ov{\Gamma}$ is of finite length.

Recall that we have a duality $\Ext^2_\Gamma(-,\Gamma):\DD \equi \Ext^2_\Gamma(\DD,\Gamma)$. On the other hand, by Proposition \ref{enough}(1), we have $\DD = \fl \un{\Gamma}$ and $\Ext^2_\Gamma(\DD,\Gamma) = \fl \ov{\Gamma}^{\op}$. Therefore we have a duality $F:\fl \ov{\Gamma}^{\op}\equi \fl \un{\Gamma}$ between two abelian categories.

Observe that $\fl \un{\Gamma}$ has an injective cogenerator $\Hom_R(\Gamma,I)$ by the duality $\Hom_R(-,I):\fl \un{\Gamma} \to \fl\un{\Gamma}^{\op}$, where $I$ is an injective hull of $R/\rad R$. Thus the abelian category $\fl \ov{\Gamma}^{\op}$ has a projective generator, which we denote by $P$.

Suppose that $\ov{\Gamma}$ is not of finite length over $R$. In particular, we have the following infinite chain of proper surjections
\[
 \cdots \defl \ov{\Gamma}/\rad^3 \ov{\Gamma} \defl \ov{\Gamma}/\rad^2 \ov{\Gamma} \defl \ov{\Gamma}/\rad \ov{\Gamma}.
\]
in $\fl \ov{\Gamma}^{\op}$. Take a surjection $f_1:P^k \defl \ov{\Gamma} / \rad \ov{\Gamma}$. Then this lifts to morphisms $f_i : P^k \to \ov{\Gamma} / \rad^i \ov{\Gamma}$. Since the kernel of $ \ov{\Gamma} / \rad^i \ov{\Gamma} \defl \ov{\Gamma} / \rad^{i-1} \ov{\Gamma}$ is contained in $\rad \ov{\Gamma} / \rad^i \ov{\Gamma}$, it follows that $f_i$ is surjective for each $i$, which contradicts the fact that $P^k$ has finite length.
\end{proof}

\begin{remark}
If $R$ is not artinian, $\EE$ does not necessarily has enough projectives. For example, consider Example \ref{veronese}. Then $\proj\Gamma$ has the exact structure corresponding to all the dotted arrows. In this exact structure, there exists no non-zero projective object.
\end{remark}

\subsection{AR conflations and the Grothendieck group}
In this subsection, we investigate the Grothendieck group of exact categories. Several papers showed that the relation of the Grothendieck group $\KKK_0(\EE)$ is generated by AR sequences when $\EE$ is a particular exact category of finite type, e.g. \cite{ar1,but,yo}. Our aim in this subsection is to unify these results.

Let $\EE$ be a Krull-Schmidt category. First we recall the following basic concepts in the AR theory. A morphism $g:Y\to Z$ in $\EE$ is called \emph{right almost split} if $g$ is not a retraction and any non-retraction $h:W \to Z$ factors through $g$. Dually we define \emph{left almost split}.
We say that a conflation $X \xrightarrow{f} Y \xrightarrow{g} Z$ in $\EE$ is an \emph{AR conflation} if $f$ is left almost split and $g$ is right almost split. We say that $\EE$ \emph{has AR conflations} if for every indecomposable non-projective object $Z$ there exists an AR conflation ending at $Z$, and for every indecomposable non-injective object $X$ there exists an AR conflation starting at $X$.
For further properties of AR conflations, we refer the reader to Appendix A.1. In Corollary \ref{exist}, we will prove that Krull-Schmidt exact category $\EE$ has AR conflations if $\EE$ is of finite type and the endomorphism ring of an additive generator of $\EE$ is noetherian.

Next we introduce some notation concerning the Grothendieck group. For a Krull-Schmidt exact category $\EE$, let $\GGG(\EE)$ be the free abelian group $\bigoplus_{[X] \in \ind \EE} \mathbb{Z}\cdot[X]$ generated by the set $\ind \EE$ of isomorphism classes of indecomposable objects in $\EE$.
We denote by $\Ex(\EE)$ the subgroup of $\GGG(\EE)$ generated by
\[
\{ [X] - [Y] + [Z] \text{ $|$ there exists a conflation $X \infl Y \defl Z$ in $\EE$} \}.
\]
We call the quotient group $\KKK_0(\EE) :=\GGG(\EE) / \Ex(\EE)$ the \emph{Grothendieck group} of $\EE$.
We denote by $\AR(\EE)$ the subgroup of $\Ex(\EE)$ generated by
\[
\{ [X] - [Y] + [Z] \text{ $|$ there exists an AR conflation $X \infl Y \defl Z$ in $\EE$} \}.
\]

Now we prove the main result about the relation of the Grothendieck groups.
\begin{theorem}\label{groth}
Let $\EE$ be a Krull-Schmidt exact category of finite type such that the endomorphism ring of an additive generator of $\EE$ is noetherian. If the exact structure on $\EE$ is admissible, then $\Ex(\EE) = \AR(\EE)$ holds.
\end{theorem}
\begin{proof}
Since $\AR(\EE) \subset \Ex(\EE)$ always holds, we will check $\Ex(\EE) \subset \AR(\EE)$. Let $\Gamma$ be the endomorphism ring of an additive generator of $\EE$. Then $\EE \equi \proj \Gamma$ holds, so we may assume that $\EE = \proj\Gamma$ for a semiperfect noetherian ring $\Gamma$.
Denote by $\SS$ the set of simple $\Gamma$-modules corresponding to the exact structure on $\EE = \proj\Gamma$ under Theorem \ref{2regularcorresp}.

Let $X_1 \xrightarrow{f_1} Y_1 \xrightarrow{g_1} Z_1$ be a conflation in $\EE$ and $M := \coker(g_1)$ in $\mod\Gamma$. Then $M$ is in $\Filt\SS$ by Theorem \ref{2regularcorresp}. We show $[X_1] - [Y_1] + [Z_1] \in \AR(\EE)$.
Suppose that there exists another conflation $X_2 \xrightarrow{f_2} Y_2 \xrightarrow{g_2} Z_2$ in $\EE$ such that $M \iso \coker(g_2)$. Then we have the exact sequences
\[
0 \to {X_i} \xrightarrow{f_i} Y_i \xrightarrow{g_i} Z_i \to M \to 0
\]
in $\mod\Gamma$ for each $i=1,2$. Thus Schanuel's lemma shows that $X_1 \oplus Y_2 \oplus Z_1 \iso X_2 \oplus Y_1 \oplus Z_2$, which implies that $[X_1] - [Y_1] + [Z_1] = [X_2] - [Y_2] + [Z_2]$ in $\GGG(\EE)$. Thus it suffices to show the following claim to prove our theorem.

\emph{Claim: For any $M \in \Filt\SS$, there exists at least one exact sequence in $\mod\Gamma$
\[
0 \to X \to Y \to Z \to M \to 0
\] with $X,Y,Z \in \proj\Gamma$ and $[X]-[Y]+[Z] \in \AR(\EE)$.}

We will show this claim by induction on $l(M)$, the length of $M$ as a $\Gamma$-module.
Suppose that $l(M) = 1$, that is, $M\in \SS$. Take the minimal projective resolution $0 \to X \xrightarrow{f} Y \xrightarrow{g} Z \to M \to 0$ of $M$ with $X,Y,Z \in \proj\Gamma = \EE$. By the proof of Proposition \ref{arconfl}, we have that $X \xrightarrow{f} Y \xrightarrow{g} Z$ is indeed an AR-conflation in $\EE$. Thus $[X] - [Y] + [Z] \in \AR(\EE)$.

Now suppose that $l(M) > 1$. Take a simple $\Gamma$-module $M_1 \in \SS$ which is a submodule of $M$. Then we have an exact sequence $0 \to M_1 \to M \to M_2 \to 0$, where $M_1$ and $M_2$ are in $\Filt \SS$. Since $l(M_1),l(M_2)<l(M)$, we have the corresponding projective resolutions $0 \to X_i \xrightarrow{f_i} Y_i \xrightarrow{g_i} Z_i \to M_i \to 0$ such that $[X_i] -[Y_i] + [Z_i] \in \AR(\EE)$ for $i=1,2$ by induction hypothesis. By the horseshoe lemma, we obtain a projective resolution $0 \to X_1 \oplus X_2 \xrightarrow{f_1 \oplus f_2} Y_1\oplus Y_2 \xrightarrow{g_1 \oplus g_2} Z_1 \oplus Z_2 \to M \to 0$. Then we have
\[
[X_1\oplus X_2 ] -[Y_1\oplus Y_2] + [Z_1 \oplus Z_2]
=\sum_{i=1,2}([X_i] -[Y_i] + [Z_i]) \in \AR(\EE),
\]
which completes the proof of the claim.
\end{proof}

\begin{corollary}\label{grocor}
Let $R$ be a noetherian complete local ring and $\EE$ a Hom-noetherian idempotent complete exact $R$-category of finite type. Suppose either $R$ is artinian or $\EE$ has enough projectives. Then $\Ex(\EE) = \AR(\EE)$ holds.
\end{corollary}
\begin{proof}
In both cases, $\EE$ is an admissible exact category by Proposition \ref{admissible}. Thus Theorem \ref{groth} applies.
\end{proof}

\section{Classifications of CM-finite algebras}
In this section, we apply our previous results to the representation theory of Iwanaga-Gorenstein algebras and orders. More generally, we study the left perpendicular category $^\perp U$ for a cotilting module $U$. \emph{Throughout this section, we fix a commutative noetherian complete local ring $R$}.

\subsection{Cotilting modules}
First, we introduce the notion of \emph{cotilting module} following \cite{applications}. Let $\Lambda$ be a noetherian ring and $U \in \mod\Lambda$ a $\Lambda$-module. We denote by $^\perp U$ the subcategory of $\mod\Lambda$ consisting of all modules $X$ satisfying $\Ext^{>0}_\Lambda(X,U)=0$. Since $^\perp U$ is an extension-closed subcategory of $\mod\Lambda$, we always regard $^\perp U$ as an exact category.
\begin{definition}\label{cotilting}
We say that $U$ is a \emph{cotilting module} if it satisfies the following conditions.
\begin{enumerate}
\item[(C1)] $\id U_\Lambda$ is finite.
\item[(C2)] $\Ext^{>0}_\Lambda(U,U) =0$.
\item[(C3)] $^\perp U$ has enough injectives $\add U$, that is, for every $X$ in $^\perp U$, there exists an exact sequence
\[
0 \to X \to U^0 \to Y \to 0
\]
in $\mod\Lambda$ with $Y \in {}^\perp U$ and $U^0 \in \add U$.
\end{enumerate}
\end{definition}

We shall see in Proposition \ref{ordercase} that if we restrict to $R$-orders, then our definition of cotilting modules coincides with the usual one, e.g. in \cite{higher}.
From our definition, the following property is immediate. Here we say that an object $X$ in an exact category $\EE$ is a \emph{projective generator} (resp. \emph{injective cogenerator}) if $\EE$ has enough projectives $\add X$ (resp. enough injectives $\add X$).
\begin{proposition}\label{cotiltingexact}
Let $\Lambda$ be a noetherian ring and $U \in \mod\Lambda$ a cotilting $\Lambda$-module. Then $^\perp U$ is an exact category with a projective generator $\Lambda$ and an injective cogenerator $U$.
\end{proposition}

Let $R$ be a Cohen-Macaulay local ring admitting a canonical module $\omega$, for example, complete Cohen-Macaulay local ring. For a noetherian $R$-algebra $\Lambda$, we denote by $\CM \Lambda$ the subcategory of $\mod\Lambda$ consisting of modules which are maximal Cohen-Macaulay as $R$-modules. A noetherian $R$-algebra $\Lambda$ is called an \emph{$R$-order} if $\Lambda\in\CM\Lambda$ holds. For an $R$-order $\Lambda$, there exists a duality $D_d:= \Hom_R(-,\omega) :\CM \Lambda \equi \CM \Lambda^{\op}$. It is immediate that $\CM \Lambda$ is an extension-closed subcategory of $\mod\Lambda$ with a projective generator $\Lambda$ and an injective cogenerator $D_d\Lambda$.

We prepare the following well-known properties of $R$-orders.
\begin{lemma}\label{wellknown}
Let $R$ be a $d$-dimensional complete Cohen-Macaulay local ring, $\Lambda$ an $R$-order and $M \in \CM\Lambda$. Then the following holds.
\begin{enumerate}
\item $\CM \Lambda = {}^\perp (D_d \Lambda)$ holds, and $\id (D_d \Lambda)_\Lambda = d$.
\item $\id M_\Lambda \geq d$ holds, and $\id M_\Lambda \leq m$ if and only if $\Ext^{>m-d}_\Lambda(X,M)=0$ for all $X \in \CM\Lambda$.
\end{enumerate}
\end{lemma}
\begin{proof}
Both follow from \cite[Proposition 1.1]{gn}.
\end{proof}
\begin{proposition}\label{ordercase}
Let $R$ be a $d$-dimensional complete Cohen-Macaulay local ring, $\Lambda$ a noetherian $R$-algebra and $U$ a finitely generated $\Lambda$-module.
\begin{enumerate}
\item Suppose that $\Lambda$ is an $R$-order and $U \in \CM\Lambda$. Then $U$ is a cotilting module if and only if $U$ satisfies the condition \upshape{(C1), (C2)} in Definition \ref{cotilting} and the following.
\begin{enumerate}
\item[\upshape{(C$3'$)}] There exists an exact sequence
\[
0 \to U_n \to \cdots \to U_1 \to U_0 \to D_d \Lambda \to 0
\]
in $\mod\Lambda$ for some $n$ such that $U_i \in \add U$ for each $i$.
\end{enumerate}
In particular, $U$ is a cotilting module with $\id U_\Lambda \leq d$ if and only if $\add U = \add D_d \Lambda$.

\item Suppose that $U$ is a cotilting $\Lambda$-module. Then $^\perp U \subset \CM \Lambda$ holds if and only if $\Lambda$ is an $R$-order and $U \in \CM\Lambda$ holds.
\end{enumerate}
\end{proposition}
\begin{proof}
(1)
Suppose that $U$ satisfies (C1), (C2) and (C$3'$). Then the same proof as in \cite[Theorem 5.4]{applications} implies that (C3) holds (see \cite[Proposition 3.2.2]{higher} for the order case).

Conversely, suppose that $U \in \CM\Lambda$ satisfies (C1)-(C3). First we show $^\perp U \subset \CM \Lambda$. Let $X$ be in $^\perp U$. Then by (C3), we have an exact sequence $0 \to X \to U^0 \to U^1 \to \cdots \to U^{d-1}$ with $U^i$ in $\add U$ for each $i$. Since each $U^i$ is maximal Cohen-Macaulay as an $R$-module, it follows from the depth lemma \cite[Proposition 1.2.9]{bh} that $X$ is in $\CM \Lambda$.

Next we will see that (C$3'$) holds. Put $\XX :={}^\perp U$ and denote by $\widehat{U}$ the subcategory of $\mod \Lambda$ consisting of modules $Y$ such that there exists an exact sequence
\[
0 \to U_n \to \cdots \to U_1 \to U_0 \to Y \to 0
\]
in $\mod\Lambda$ for some $n$ with $U_i \in \add U$ for each $i$. Then the Auslander-Buchweitz theory implies that $Y \in \widehat{U}$ if $\Ext_\Lambda^{>0}(\XX,Y)=0$ (see \cite[Proposition 3.6]{mcm} or \cite[Corollary A.3]{en} for the detail). Since $\XX \subset \CM \Lambda$ and $\Ext_\Lambda^{>0}(\CM \Lambda,D_d \Lambda) = 0$, we obtain $D_d \Lambda \in \widehat{U}$, which implies (C$3'$).

It follows from Lemma \ref{wellknown}(2) that $\id U_\Lambda \leq d$ if and only if $U$ is an injective object in $\CM\Lambda$. Thus the remaining assertion easily follows from the definition and (C$3'$).

(2)
If we have $^\perp U \subset \CM\Lambda$, then in particular $\Lambda$ and $U$ are in $\CM\Lambda$, which in particular implies that  $\Lambda$ is an $R$-order. Thus the ``only if'' part follows. The the  ``if'' part has already shown in the proof of (1).
\end{proof}

A noetherian ring $\Lambda$ is called \emph{Iwanaga-Gorenstein} if both $\id(\Lambda_\Lambda)$ and $\id({}_\Lambda \Lambda)$ are finite. By \cite[Lemma A]{za}, we have $\id (\Lambda_\Lambda) = \id ({}_\Lambda \Lambda)$ in this case.
For an Iwanaga-Gorenstein ring $\Lambda$, a $\Lambda$-module $X \in \mod\Lambda$ is called \emph{Gorenstein-projective} if $X$ is in $\GP \Lambda := {}^\perp \Lambda$. Then $\GP \Lambda$ is a Frobenius exact category with a projective generator $\Lambda$. An Iwanaga-Gorenstein ring $\Lambda$ is \emph{GP-finite} if $\GP \Lambda$ is of finite type. These concepts are special cases of cotilting modules, as the following proposition shows. If $\Lambda$ is an $R$-order, then this can be easily proved by using Proposition \ref{ordercase}.

\begin{proposition}\label{gorencotilt}
Let $\Lambda$ be a noetherian ring.
\begin{enumerate}
\item $\Lambda$ is Iwanaga-Gorenstein if and only if $\Lambda_\Lambda$ is a cotilting $\Lambda$-module.
\item Let $U \in \mod\Lambda$ be a cotilting $\Lambda$-module. Then $^\perp U$ is Frobenius if and only if $\Lambda$ is Iwanaga-Gorenstein and $\add U = \proj \Lambda$ holds.
\end{enumerate}
\end{proposition}
\begin{proof}
(1)
First we show the ``only if'' part. Suppose that $\Lambda$ is Iwanaga-Gorenstein. Then $\Lambda_\Lambda$ clearly satisfies (C1) and (C2). Moreover (C3) follows from the fact that $^\perp \Lambda = \GP \Lambda$ is the Frobenius category with an injective cogenerator $\Lambda$.

To show the ``if'' part, we use the main result of \cite{hh}. We refer the reader to \cite{hh} and references therein for the unexplained concepts. In \cite{hh}, it was shown that $\Lambda$ is Iwanaga-Gorenstein if every module in $\mod\Lambda$ has a finite Gorenstein-projective dimension. Thus it suffices to see that this is the case if $\Lambda_\Lambda$ is cotilting, which follows from \cite[Proposition 4.2]{en}.

(2)
Clear from (1) and Proposition \ref{cotiltingexact}.
\end{proof}

In the paper \cite{en}, it was characterized when a given exact category is exact equivalent to the exact category of the form $^\perp U$ for a cotilting module $U$ over a noetherian ring. For an integer $n\geq 1$, we say that $\EE$ \emph{has $n$-kernels} if for every morphism $f:X\to Y$ in $\EE$, there exists a complex
 \[
 0 \to X_n \xrightarrow{f_n} \cdots \xrightarrow{f_2} X_1 \xrightarrow{f_1} X \xrightarrow{f} Y
 \]
  in $\EE$ such that the following diagram is exact.
  \[
 0 \to \EE(-,X_n) \xrightarrow{\EE(-,f_n)}  \cdots \xrightarrow{\EE(-,f_2)}  \EE(-,X_1) \xrightarrow{\EE(-,f_1)} \EE(-,X) \xrightarrow{\EE(-,f)}  \EE(-,Y)
  \]

\begin{proposition}[{\cite[Corollary 4.12]{en}}]\label{eno}
Suppose that $\EE:= (\proj \Gamma,F)$ is an exact category for a noetherian ring $\Gamma$ and $n \geq 2$ is an integer.
\begin{enumerate}
\item The following are equivalent.
\begin{enumerate}
\item There exist a noetherian ring $\Lambda$ and a cotilting $\Lambda$-module $U$ with $\id U_\Lambda \leq n$ such that $^\perp U$ is exact equivalent to $\EE$.
\item $\EE$ has projective generators and injective cogenerators, and $\gl \Gamma \leq n$ holds.
\item $\EE$ has projective generators, injective cogenerators, and $(n-1)$-kernels.
\end{enumerate}

\item
Suppose that the condition in \upshape{(1)} holds. Take a projective generator $P$ and an injective cogenerator $I$, and put $\Lambda := \End_\EE(P)$ and $U := \EE(P,I) \in \mod\Lambda$. Then $\Lambda$ and $U$ satisfies \upshape{(1)} and $\EE(P,-)$ gives an exact equivalence $\EE \equi {}^\perp U$.
\end{enumerate}
\end{proposition}
\begin{proof}
Note that the endomorphism ring of every object in $\EE$ is noetherian, see e.g. \cite[Proposition 2.3]{sand}. Thus \cite[Theorem 4.11]{en} applies.
\end{proof}

\subsection{Classifications for noetherian $R$-algebras}
To state our classifications, it is convenient to introduce the following terminology.
\begin{definition}
Let $R$ be a noetherian complete local ring.
\begin{enumerate}
\item
We say that a pair $(\Lambda,U)$ is an \emph{$n$-cotilting pair} if $\Lambda$ is a noetherian $R$-algebra and $U \in \mod\Lambda$ is a cotilting $\Lambda$-module with $\id U_\Lambda \leq n$.
Cotilting pairs $(\Lambda_1,U_1)$ and $(\Lambda_2,U_2)$ are said to be \emph{equivalent} if there exists an equivalence $\mod \Lambda_1 \to \mod \Lambda_2$ which induces an equivalence $\add U_1 \to \add U_2$.
\item
We say that a pair $(\Gamma, \XXX)$ is an \emph{algebra with dotted arrows} if $\Gamma$ is a noetherian $R$-algebra and $\XXX$ is a set of dotted arrows of $Q(\Gamma)$ (see Definition \ref{const}). Two such pairs $(\Gamma_1,\XXX_1)$ and $(\Gamma_2,\XXX_2)$ are said to be \emph{equivalent} if there exists an equivalence $\proj \Gamma_1 \to \proj\Gamma_2$ such that $\XXX_1$ corresponds to $\XXX_2$ under the isomorphism $Q(\Gamma_1) \to Q(\Gamma_2)$. For an algebra with dotted arrows $(\Gamma, \XXX)$, let $P_\XXX$ be the direct sum of indecomposable projective $\Gamma$-modules which are not sources of dotted arrows in $\XXX$. We fix an idempotent $e_\XXX \in\Gamma$ satisfying $e_\XXX \Gamma = P_\XXX$.
\end{enumerate}
\end{definition}
The following main theorem classifies an $n$-cotilting pair of finite type by algebras with finite global dimension and set of dotted arrows.
\begin{theorem}\label{main1}
Let $R$ be a noetherian complete local ring. There exists a bijection between the following for $n \geq 2$.
\begin{enumerate}
\item Equivalence classes of $n$-cotilting pairs $(\Lambda,U)$ such that $^\perp U$ is of finite type.
\item Equivalence classes of algebras with dotted arrows $(\Gamma, \XXX)$ such that $\gl \Gamma \leq n$ and $\Gamma/\Gamma e_\XXX \Gamma$ is of finite length over $R$.
\item Exact equivalence classes of Hom-noetherian idempotent complete exact $R$-category $\EE$ of finite type such that $\EE$ has enough projectives, enough injectives and $(n-1)$-kernels.
\end{enumerate}
\end{theorem}
\begin{proof}
By Corollary \ref{dotted}, we have the bijection between (2) and the following.
\begin{enumerate}
\item[(4)]
Exact equivalence classes of $(\proj \Gamma, F)$, where $\Gamma$ is a noetherian $R$-algebra with $\gl \Gamma \leq n$ and $F$ is an exact structure on $\proj\Gamma$ with enough projectives and injectives.
\end{enumerate}
In fact, Proposition \ref{admissible} shows that $F$ in (4) is admissible, and Proposition \ref{enough} and Corollary \ref{projinj} implies that $F$ has enough projectives and injectives if and only if $\Gamma / \Gamma e_\XXX \Gamma$ is of finite length.

By Proposition \ref{eno}, we have the bijections between (2), (3) and (4). Finally, we have maps between (1) and (3) as follows.
For an $n$-cotilting pair $(\Lambda, U)$ in (1), we put $\EE := {}^\perp U$.
For an exact category $\EE$ in (3), there exists an $n$-cotilting pair $(\Lambda,U)$ such that $^\perp U$ is exact equivalent to $\EE$ by Proposition \ref{eno}, which satisfies the condition of (1).
The straightforward argument shows that these maps are mutually inverse to each other.
\end{proof}

Next we shall apply Theorem \ref{main1} to GP-finite Iwanaga-Gorenstein algebras. For a translation quiver $Q$, we consider a \emph{$\tau$-orbit} (that is, a connected component of the graph $Q'$ consisting of the same vertices as $Q$ and all the dotted translation arrows of $Q$). We say that a $\tau$-orbit is \emph{stable} if every vertex on it is both a source and a target of dotted arrows.
\begin{corollary}\label{main2}
Let $R$ be a noetherian complete local ring. There exists a bijection between the following for $n \geq 2$.
\begin{enumerate}
\item Morita equivalence classes of GP-finite Iwanaga-Gorenstein noetherian $R$-algebras $\Lambda$ with $\id \Lambda \leq n$.
\item Equivalence classes of algebras with dotted arrows $(\Gamma,\XXX)$ satisfying the following.
\begin{enumerate}
\item $\gl \Gamma \leq n$.
\item $\Gamma/\Gamma e_\XXX \Gamma$ is of finite length over $R$.
\item $\XXX$ is a union of stable $\tau$-orbits in $Q(\Gamma)$.
\end{enumerate}
\end{enumerate}
\end{corollary}
\begin{proof}
By Proposition \ref{gorencotilt}, it suffices to show that the pair $(\Gamma, F)$ in Theorem \ref{main1}(2) gives the Frobenius exact structure on $\proj\Gamma$ if and only if $\XXX$ in Theorem \ref{main1}(2) is a union of stable $\tau$-orbits. Since the exact category $\proj \Gamma$ has enough projectives and injectives by Proposition \ref{enough}(2) and Corollary \ref{projinj}, the exact structure $F$ is Frobenius if and only if the class of projectives and that of injectives coincide. For a set of dotted arrows $\XXX$, an indecomposable object $M$ in $\proj \Gamma$ is not projective (resp. not injective) in $F$ if and only if there exists a dotted arrow in $\XXX$ starting at (resp. ending at) $M$. Therefore the class of indecomposable non-projective objects and that of indecomposable non-injective objects coincide if and only if $\XXX$ is a union of some stable $\tau$-orbits, which clearly implies the assertion.
\end{proof}

We refer the reader to Example \ref{example1} in Section 1 for the detailed example of this classification.

\subsection{Classifications for $R$-orders}
Restricting Theorem \ref{main1} to the case of $R$-orders, we obtain the corresponding result as follows.
\begin{corollary}\label{ordercor}
Let $R$ be a complete Cohen-Macaulay local ring. Then there exists a bijection between the following for $n \geq 2$.
\begin{enumerate}
\item Equivalence classes of $n$-cotilting pairs $(\Lambda, U)$ such that $\Lambda$ is an $R$-order, $U$ is in $\CM\Lambda$ and $^\perp U$ is of finite type.
\item Equivalence classes of algebras with dotted arrows $(\Gamma, \XXX)$ satisfying the following.
\begin{enumerate}
\item $\gl \Gamma \leq n$.
\item $\Gamma / \Gamma e_\XXX \Gamma$ is of finite length over $R$.
\item $\Gamma e_\XXX$ is maximal Cohen-Macaulay as an $R$-module.
\end{enumerate}
\end{enumerate}
\end{corollary}
\begin{proof}
We check that the bijections in Theorem \ref{main1} restricts to this case. By Proposition \ref{ordercase}(2), we may replace the equivalence class (1) with
\begin{enumerate}
\item[(1$'$)] Equivalence classes of $n$-cotilting pairs $(\Lambda, U)$ such that $^\perp U$ is of finite type and $^\perp U \subset \CM \Lambda$ holds.
\end{enumerate}

Suppose that $(\Lambda,U)$ and $(\Gamma, \XXX)$ correspond to each other under Theorem \ref{main1}. It suffices to show that $^\perp U \subset \CM \Lambda$ if and only if $\Gamma e_\XXX$ is maximal Cohen-Macaulay as an $R$-module.
Notice that $\add e_\XXX \Gamma$ is the category of all projective objects in the exact category $(\proj \Gamma,F)$, thus $e_\XXX \Gamma$ is a projective generator of it. Hence we may assume that $\Lambda = \End_\Gamma (e_\XXX \Gamma) = e_\XXX \Gamma e_\XXX$. In this situation, we have an embedding $\Hom_\Gamma(e_\XXX \Gamma, -) :\proj \Gamma \to \mod\Lambda$ whose essential image is $^\perp U$. It follows that $^\perp U = \add \Hom_\Gamma(e_\XXX \Gamma, \Gamma) = \add \Gamma e_\XXX$. Therefore $\Gamma e_\XXX$ is maximal Cohen-Macaulay as an $R$-module if and only if $^\perp U \subset \CM\Lambda$ holds.
\end{proof}

This gives the following Auslander-type correspondence for Cohen-Macaulay-finite $R$-orders with $\dim R \geq 2$. It largely improves \cite[Theorem 4.2.3]{higher} for the case $n=1$, because our result gives a bijection for Cohen-Macaulay-finite orders and some assumptions on $\Gamma$ in \cite[Theorem 4.2.3]{higher} are dropped.
\begin{corollary}\label{arorder}
Let $R$ be a complete Cohen-Macaulay local ring with $\dim R =d  \geq 2$. Then there exists a bijection between the following.
\begin{enumerate}
\item Morita equivalence classes of $R$-orders $\Lambda$ such that $\CM \Lambda$ is of finite type.
\item Equivalence classes of algebras with dotted arrows $(\Gamma, \XXX)$ satisfying the following.
\begin{enumerate}
\item $\gl \Gamma = d$.
\item $\Gamma / \Gamma e_\XXX \Gamma$ is of finite length over $R$.
\item $\Gamma e_\XXX$ is maximal Cohen-Macaulay as an $R$-module.
\end{enumerate}
\end{enumerate}
\end{corollary}
\begin{proof}
Let us apply Corollary \ref{ordercor} to the case $n=d$.
By Proposition \ref{ordercase}, it suffices to show that for a noetherian $R$-algebra $\Gamma$ satisfying the conditions of (2)(c), we have $\gl \Gamma \geq d$. Every noetherian $R$-algebra $\Gamma$ satisfies $\dim_R \Gamma \leq \gl \Gamma$ , see e.g. \cite[Corollary 3.5(4)]{towards}. Since $\Gamma e_\XXX$ is a direct summand of $\Gamma$ and $\Gamma e_\XXX$ is maximal Cohen-Macaulay, we have $d = \dim_R \Gamma e_\XXX \leq \dim_R \Gamma \leq \gl \Gamma$.
\end{proof}

\appendix
\section{The AR theory for exact categories over a noetherian ring}
In this appendix, we study the Auslander-Reiten theory for exact categories over a noetherian complete local ring. In particular, we investigate the relationship between the existence of AR conflations, the AR duality and the notion of dualizing $R$-varieties.

\subsection{Existence of AR conflations}
Let $\EE$ be a Krull-Schmidt category. We denote by $\JJ$ the Jacobson radical of $\EE$.
We need later the following easy lemma about kernel-cokernel pairs in Krull-Schmidt categories.

\begin{lemma}\label{kc}
Suppose that $X \to Y \to Z
$ is a kernel-cokernel pair in $\EE$. Then this complex is isomorphic to the direct sum of kernel-cokernel pairs of the following forms.
\begin{enumerate}
\item $A = A \to 0$ for some A in $\EE$.
\item $0\to B=B$ for some B in $\EE$.
\item $X'\to Y' \to Z'$ with all morphisms in $\JJ$.
\end{enumerate}
\end{lemma}

Let $\EE$ be a Krull-Schmidt exact category. Recall that a conflation $X \xrightarrow{f} Y \xrightarrow{g} Z$ in $\EE$ is an \emph{AR conflation} if $f$ is left almost split and $g$ is right almost split.
It immediately follows from the definition that if $X \xrightarrow{f} Y \xrightarrow{g} Z$ is an AR conflation, then $f$ is left minimal, that is, any $\varphi : Y \to Y$ satisfying $\varphi f = f$ is an automorphism. Dually $g$ is right minimal in this case. Also one can prove the uniqueness of the AR conflation as in the classical case.

The following classical observation is useful for us.
For an object $X$ in $\EE$, we put $S_X := P_X / \JJ(-,X) \in \Mod \EE$ and $S^X := P^X /\JJ(X,-) \in \Mod \EE^{\op}$.

\begin{proposition}\label{ras}
Let $\EE$ be a Krull-Schmidt category. Then the following hold.
\begin{enumerate}
\item The map $X \mapsto S_X$ gives a bijection between $\ind\EE$ and the set of isomorphism classes of simple object in $\Mod\EE$.
\item A morphism $Y \xrightarrow{g} Z$ is right almost split if and only if $Z$ is indecomposable and $P_Y \xrightarrow{P_g} P_Z \to S_Z \to 0$ is exact.
\item There exists a right almost split morphism ending at $Z$ if and only if $S_Z$ is finitely presented.
\end{enumerate}
\end{proposition}
\begin{proof}
It is well-known, see \cite{artin2} for example.
\end{proof}

The following proposition says that the existence of the right almost split map ensures the existence of the AR conflation.
\begin{proposition}\label{arconfl}
Let $Z$ be an indecomposable non-projective object in $\EE$. Then the following are equivalent.
\begin{enumerate}
\item There exists a right almost split morphism to $Z$.
\item There exists an AR conflation $X \infl Y \defl Z$ in $\EE$ ending at $Z$.
\end{enumerate}
\end{proposition}
\begin{proof}
(1) $\Rightarrow$ (2):
Since $Z$ is not projective, there exists some object $M$ in $\DD$ such that $M(Z) \neq 0$ by Proposition \ref{proj}. This means that there exists some non-zero morphism $a:P_Z \to M$ by the Yoneda lemma. Since $\im a$ is finitely generated and $\DD$ is a Serre subcategory of  $\mod\EE$, we have that $\im a$ is also in $\DD$.
Since $P_Z$ has a unique maximal subobject $\JJ(-,Z)$, it is easy to see that the cokernel of the composition $\JJ(-,Z) \to P_Z \to \im a$ is isomorphic to $S_Z$. Since $\JJ(-,Z)$ is finitely generated by (1), we have $S_Z \in \DD$.
Therefore, we obtain a conflation $X \xrightarrow{f} Y \xrightarrow{g} Z$ such that $g$ is right almost split. Now we may assume that $f$ is in $\JJ$ by Lemma \ref{kc}. This clearly implies that $g$ is right minimal and the classical argument shows that $f$ is left almost split.

(2) $\Rightarrow$ (1):
Obvious from the definition.
\end{proof}
We say that $\EE$ \emph{has AR conflations} if for every indecomposable non-projective object $Z$ there exists an AR conflation ending at $Z$, and for every indecomposable non-injective object $X$ there exists an AR conflation starting at $X$.
\begin{corollary}\label{exist}
Suppose that $\EE := \proj \Lambda$ is an exact category for a semiperfect noetherian ring $\Lambda$. Then $\EE$ has AR conflations.
\end{corollary}
\begin{proof}
We have $\Mod \EE \equi \Mod \Lambda$, so every simple object in this category is finitely presented because $\Lambda$ is noetherian. Thus the assertion holds by Propositions \ref{ras} and \ref{arconfl}.
\end{proof}

\subsection{AR conflations, the AR duality and dualizing varieties}

Next we investigate the relationship between AR conflations and the AR duality in the general setting. \emph{In what follows, we denote by $R$ a commutative noetherian complete local ring}.

We denote by $I$ the injective envelope of the simple $R$-module $R/ \rad R$, and by $D:\Mod R \to \Mod R$ the Matlis dual $D=\Hom_R(-,I)$. It is well-known that this induces a duality $D:\noeth R \leftrightarrow \art R$ and $D: \fl R \leftrightarrow \fl R$, where $\noeth R$, $\art R$ and $\fl R$ denotes the subcategory of $\Mod R$ consisting of $R$-modules which are noetherian, artinian and of finite length respectively.

\emph{From now on, we assume that $\EE$ is a Hom-noetherian idempotent complete exact $R$-category.}
Recall that a morphism $f:W \to Z$ in $\EE$ is \emph{projectively trivial} if for each object $X$, the induced map $\Ext^1_\EE(Z,X) \to \Ext^1_\EE(W,X)$ is zero. It is easy to see that $f$ is projectively trivial if and only if $f$ factors through every deflation $Y \defl Z$. If $\EE$ has enough projectives, then $f$ is projectively trivial if and only if $f$ factors through some projective object.
Denote by $\PP(X,Y)$ the set of all projectively trivial morphisms from $X$ to $Y$. Then $\PP$ is a two-sided ideal of $\EE$ and we put $\un{\EE} := \EE/\PP$, which we call the \emph{projectively stable category}.
Dually we define the notion of \emph{injectively trivial} morphisms and the \emph{injectively stable category} $\ov{\EE}:= \EE/\II$.

The following theorem is an exact category version of \cite[Proposition I.2.3]{rvdb}, and it generalizes \cite[Theorem 3.6]{lnp} and \cite[Proposition 2.4]{jiao}.
Also see \cite{lnp,in} for related work.
Note that we do not assume that $R$ is artinian or $\EE$ is Hom-finite. This enables us to give another proof of the one implication of Auslander's theorem about isolated singularities, see Remark \ref{isolatedrmk}
\begin{proposition}\label{ardual}
Suppose that $\EE$ is Hom-noetherian Krull-Schmidt exact $R$-category such that either (i) $\EE$ is Ext-noetherian or (ii) $\ov{\EE}$ is Hom-finite. Let $Z$ be an indecomposable non-projective object in $\EE$. Then the following are equivalent.
\begin{enumerate}
\item There exists a right almost split morphism to $Z$.
\item There exists an AR conflation $X \infl Y \defl Z$ ending at $Z$.
\item The functor $D \Ext^1_\EE(Z,-):\ov{\EE} \to \Mod R$ is representable.
\end{enumerate}
Moreover, if $X \infl Y \defl Z$ is an AR conflation, then $D \Ext^1_\EE(Z,-) \equi \ov{\EE}(-,X)$ holds, and thus $\Ext^1_\EE(Z,W)$ and $\ov{\EE}(W,X)$ are of finite length over $R$ for every $W \in \EE$.
\end{proposition}
\begin{proof}
(1) $\Leftrightarrow$ (2): This is Proposition \ref{arconfl}.

(2) $\Rightarrow$ (3):
Suppose that $\eta: X \infl Y \defl Z$ is an AR conflation. We regard $0 \neq \eta \in \Ext^1_\EE(Z,X)$. Since $I$ is a cogenerator of $\Mod R$, there exists a non-zero morphism $\gamma:\Ext^1_\EE(Z,X) \to I$ in $\Mod R$ such that $\gamma(\eta) \neq 0$. Then we have a morphism
\[
\langle -,-\rangle _W : \ov{\EE}(W,X) \otimes_R \Ext^1_\EE(Z,W) \to I, \hspace{5mm} (\ov{h},\mu) \mapsto \gamma(h \mu)
\]
in $\Mod R$, where $h \mu$ is the image of $\mu$ by the induced morphism $\Ext^1_\EE(Z,f):\Ext^1_\EE(Z,W) \to \Ext^1_\EE(Z,X)$.
Then the same proof as in \cite[Lemma 2.1]{jiao} applies here to show that $\langle -,- \rangle_W$ is non-degenerate in both variables. Therefore, we have two injections
$\ov{\EE}(W,X) \hookrightarrow D \Ext^1_\EE(Z,W)$ and
$\Ext^1_\EE(Z,W) \hookrightarrow D \ov{\EE}(W,X)$ in $\Mod R$.
They are obviously isomorphisms if (ii) $\ov{\EE}$ is Hom-finite. Suppose that (i) $\EE$ is Ext-noetherian. Then $D \Ext^1_\EE(Z,W)$ is artinian, thus $\ov{\EE}(W,X)$ is of finite length. Therefore two injections are isomorphisms and both of $\ov{\EE}(W,X)$ and $\Ext^1_\EE(Z,W)$ are of finite length over $R$.
The naturality in $W$ is clear from a direct calculation, so we obtain an isomorphism of functors $D \Ext^1_\EE(Z,-) \equi \ov{\EE}(-,X)$.

(3) $\Rightarrow$ (2):
The proof of \cite[Lemma 2.2]{jiao} applies here.
\end{proof}

In the rest of the appendix, we show that the existence of AR conflations are closely related to the notion of dualizing $R$-varieties, introduced in \cite{stable} and \cite{periodic}.
\begin{definition}
Let $\CC$ be a Hom-finite idempotent complete $R$-category. We say that $\CC$ is a \emph{dualizing $R$-variety} if it satisfies both of the following.
\begin{enumerate}
\item For any finitely presented $R$-functor $F:\CC^{\op} \to \fl R$, the composition functor $DF : \CC^{\op} \to \fl R \to \fl R$ is finitely presented.
\item For any finitely presented $R$-functor $G:\CC \to \fl R$, the composition functor $DF : \CC \to \fl R \to \fl R$ is finitely presented.
\end{enumerate}
\end{definition}
The following properties are immediate, which we state without proofs.
\begin{lemma}
Let $\CC$ be a dualizing $R$-variety. Then the following hold.
\begin{enumerate}
\item The category of finitely presented functors $\mod_1 \CC$ and $\mod_1 \CC^{\op}$ are abelian categories.
\item $D$ induces a duality $D:\mod_1\CC \leftrightarrow \mod_1 \CC^{\op}$.
\end{enumerate}
\end{lemma}

The following technical lemma is elementary.
\begin{lemma}\label{ff}
Let $\EE$ be an exact category with enough projectives and injectives. Then $\mod_1 \un{\EE}$ is an abelian category with enough projectives and injectives. The functor $\un{\EE} \to \mod_1 \un{\EE}$ defined by $X \mapsto \un{\EE}(-,X)$ gives the equivalence between $\un{\EE}$ and the category of projective objects in $\mod_1\un{\EE}$, and the functor $\ov{\EE} \to \mod_1 \un{\EE}$ defined by $X \mapsto \Ext^1_\EE(-,X)$ gives the equivalence between $\ov{\EE}$ and the category of injective objects in $\mod_1\un{\EE}$.
\end{lemma}
\begin{proof}
To show that $\mod_1 \un{\EE}$ is abelian, it suffices to show that $\un{\EE}$ has weak kernels, see \cite{aus66}.
Let $g:Y \to Z$ be any morphism in $\EE$. Since $\EE$ has enough projectives, we have the following pullback diagram with $P$ being projective.
\begin{equation}\label{pbdiag}
\xymatrix{
Z' \ar@{=}[d] \inflr & E \deflr^{a} \ar[d]^{b} & Y \ar[d]^{g} \\
Z' \inflr & P \deflr^{c} & Z
}
\end{equation}
It is straightforward to check that $a$ gives the weak kernel of $\un{g} \in \un{\EE}(Y,Z)$.

The remaining assertion follows from the same argument as in \cite[Theorem 2.2.2(1)]{higher}.
\end{proof}

Now we are in position to prove our main results in this appendix, which is an exact category version of \cite[Proposition 2.2]{periodic}.
\begin{theorem}\label{confldual}
Let $\EE$ be a Hom-noetherian idempotent complete exact $R$-category with enough projectives and injectives. The following are equivalent.
\begin{enumerate}
\item The category $\EE$ has AR conflations.
\item $\un{\EE}$ is a dualizing $R$-variety.
\item $\ov{\EE}$ is a dualizing $R$-variety.
\item There exist mutually inverse equivalences $\tau: \un{\EE} \rightleftarrows \ov{\EE} : \tau^-$ such that we have natural isomorphisms $D \un{\EE}(\tau^- X,Z) \iso \Ext^1_\EE(Z,X) \iso D \ov{\EE}(X,\tau Z)$.
\end{enumerate}
In this case, $\EE$ is Ext-finite and $\un{\EE}$ and $\ov{\EE}$ are Hom-finite.
\end{theorem}

\begin{proof}
(1) $\Rightarrow$ (2):
First note that $\un{\EE}$ is Hom-finite by Proposition \ref{ardual}. Let $F\in \mod_1 \un{\EE}$ be a finitely presented functor. We have an exact sequence
$
\un{\EE}(-,Y) \to \un{\EE}(-,Z) \to F \to 0
$ in $\Mod \un{\EE}$, which is induced from $g:Y \to Z$ in $\EE$ by the Yoneda lemma.
Consider the diagram (\ref{pbdiag}). It is standard that the right square gives a conflation $E \xrightarrow{^t[b,-a]} P \oplus Y \xrightarrow{[c,g]} Z$ (see e.g. \cite[Proposition 2.12]{buhler}). Hence by replacing $g$ by $[c,g]$, we may assume that $X \xrightarrow{f} Y \xrightarrow{g} Z$ is a conflation with
$
\un{\EE}(-,Y) \xrightarrow{(-)\circ g} \un{\EE}(-,Z) \to F \to 0$
 being exact.

By the long exact sequence of Ext,
\[
\un{\EE}(-,Y) \xrightarrow{(-)\circ g} \un{\EE}(-,Z) \to \Ext^1_\EE(-,X) \xrightarrow{\Ext^1_\EE(-,f)}  \Ext^1_\EE(-,Y)
\]
 is exact in $\Mod \un{\EE}$. Thus we have an exact sequence \[
 0 \to F \to \Ext^1_\EE(-,X) \xrightarrow{\Ext^1_\EE(-,f)}  \Ext^1_\EE(-,Y).\]
Since $D :\Mod\un{\EE} \to \Mod \un{\EE}^{\op}$ is a contravariant exact functor, we obtain an exact sequence
\[
D \Ext^1_\EE(-,Y) \to D \Ext^1_\EE(-,X) \to D F \to 0.\]
By (1) and Proposition \ref{ardual}, the functors $D \Ext^1_\EE(-,Y)$ and $D \Ext^1_\EE(-,X)$ are representable, which shows that $D F$ is finitely presented.

Next suppose that $G \in \mod_1 \un{\EE}^{\op}$ is a finitely presented functor. We have an exact sequence
$
\un{\EE}(B,-) \to \un{\EE}(A,-) \to G \to 0$
 in $\Mod\un{\EE}^{\op}$. Applying $D$, we obtain an exact sequence
$
 0 \to DG \to D \un{\EE}(A,-) \to D \un{\EE}(B,-)
$
 in $\Mod \un{\EE}$. We will show that $D \un{\EE}(A,-) \equi \Ext^1_\EE(-,\tau A)$ for some $\tau A$ in $\EE$. Without loss of generality, we may assume that $A$ has no non-zero projective summands. Then we have a direct sum of AR conflations $\tau A \infl E \defl A$ in $\EE$. By Proposition \ref{ardual}, $D \un{\EE}(A,-) \equi \Ext^1_\EE(-,\tau A)$ holds. Thus we obtain the following exact sequence in $\Mod \un{\EE}$.
  \[0 \to DG \to \Ext^1_\EE(-,\tau A) \to \Ext ^1_\EE(-,\tau B)\]
Since $\un{\EE}$ has weak kernels by Lemma \ref{ff}, the subcategory $\mod_1 \un{\EE}$ is closed under kernels in $\Mod\un{\EE}$ (see \cite{aus66}). Also Lemma \ref{ff} shows that $\Ext^1_\EE(-,\tau A)$ is finitely presented. Thus $DG$ is also finitely presented.

(2) $\Rightarrow$ (4):
Since $\un{\EE}$ is a dualizing $R$-variety, we have a duality $D:\mod_1 \un{\EE}^{\op} \equi \mod_1 \un{\EE}$ between abelian categories. It induces a duality $D: \PP \equi \II$, where $\PP$ (resp. $\II$) denotes the category of projective objects in $\mod_1 \un{\EE}$ (resp. injective objects in $\mod_1 \un{\EE}^{\op}$).
By Lemma \ref{ff}, we have a equivalence $\II \equi \ov{\EE}$ and the contravariant Yoneda embedding $\un{\EE} \equi \PP$. Define the equivalence $\tau$ by the composition $\un{\EE} \equi \PP \equi \II \equi \ov{\EE}$ and denote by $\tau^-$ its quasi-inverse. Then it follows from Lemma \ref{ff} that (4) holds.

(4) $\Rightarrow$ (1):
Obvious from Proposition \ref{ardual}.

By duality, (3) is also equivalent to all the other conditions.
\end{proof}

This theorem gives an application about the relation between AR conflations and isolated singularities shown in \cite{isolated}. For a complete Cohen-Macaulay local ring $R$, we say that an $R$-order $\Lambda$ \emph{has at most an isolated singularity} if $\gl \Lambda\otimes_R R_\mathfrak{p} = \height \mathfrak{p}$ holds for every non-maximal prime ideal $\mathfrak{p}$ of $R$.
It is well-known that an $R$-order $\Lambda$ has at most an isolated singularity if and only if $\un{\CM}\: \Lambda$ is Hom-finite.
\begin{remark}\label{isolatedrmk}
Suppose that $\CM \Lambda$ has AR conflations. Then $\un{\CM}\: \Lambda$ is a dualizing $R$-variety by Theorem \ref{confldual}, hence in particular $\un{\CM}\: \Lambda$ is Hom-finite, and therefore $\Lambda$ has at most an isolated singularity. This gives a simple conceptual proof of the ``only if'' part of the main theorem of \cite{isolated}; $\CM \Lambda$ has AR conflations if and only if $\Lambda$ has at most an isolated singularity.
\end{remark}

\medskip\noindent
{\bf Acknowledgement.}
The author would like to express his deep gratitude to his supervisor Osamu Iyama for his support and many helpful comments, especially about noetherian algebras and orders. The author also thanks the anonymous referee for several helpful comments and suggestions. This work is supported by JSPS KAKENHI Grant Number JP18J21556.

\end{document}